\documentclass[11pt, reqno]{amsart}
\usepackage[margin=3cm]{geometry}
\usepackage{graphicx} 
\usepackage{adjustbox}
\usepackage{amsmath,amssymb,amsxtra,tikz,stackengine,bm,mathrsfs,bbm,tikz-cd,comment,hyperref,combelow,amsopn}
\usetikzlibrary{calc,babel}

\tikzset{
  curve/.style={
    settings={#1},
    to path={
      (\tikztostart)
      .. controls ($(\tikztostart)!\pv{pos}!(\tikztotarget)!\pv{height}!270:(\tikztotarget)$)
      and ($(\tikztostart)!1-\pv{pos}!(\tikztotarget)!\pv{height}!270:(\tikztotarget)$)
      .. (\tikztotarget)\tikztonodes
    },
  },
  settings/.code={%
    \tikzset{quiver/.cd,#1}%
    \def\pv##1{\pgfkeysvalueof{/tikz/quiver/##1}}%
  },
  quiver/.cd,
  pos/.initial=0.35,
  height/.initial=0,
}

\usepackage{enumitem}

\usepackage{mathtools}

\numberwithin{equation}{section}

\newtheorem{theorem}{Theorem}[section]
\newtheorem{proposition}[theorem]{Proposition}
\newtheorem{corollary}[theorem]{Corollary}
\newtheorem{lemma}[theorem]{Lemma}

\theoremstyle{definition}
\newtheorem{remark}[theorem]{Remark}

\newtheorem{definition}[theorem]{Definition}
\newtheorem{example}[theorem]{Example}

\newcommand{\C}{\mathbb{C}}

\newcommand{\Z}{\mathbb{Z}}

\newcommand{\Aqt}{\mathbb{A}_{q,t}}
\newcommand{\Bqt}{\mathbb{B}_{q,t}}
\newcommand{\Eqt}{\mathbb{E}_{q,t}}
\newcommand{\DBqt}{\mathbb{DB}_{q,t}}

\newcommand{\sq}{\square}

\newcommand{\adj}{\Theta}
\newcommand{\e}{\mathbbm{1}}
\newcommand{\Sym}{\mathrm{Sym}}

\DeclareMathOperator{\End}{End}

\newcommand{\AH}{\mathrm{AH}}

\newcommand{\EHA}{\mathcal{E}_{q,t}}

\newcommand{\ext}{\mathrm{ext}}

\newcommand{\z}{\mathbf{z}} 
\newcommand{\w}{\mathbf{w}}

\def\opp{\mathrm{opp}}
\def\modd{\operatorname{-mod}}

\newcommand{\Ad}{\mathrm{Ad}}
\newcommand{\ad}{\mathrm{ad}}

\newcommand{\newword}[1]{\emph{\textbf{#1}}}

\author{Nicolle Gonz\'alez}
\address{Department of Mathematics,  University of California Berkeley , CA 94720-3840, U.S.A.}
\email{nicolle@math.berkeley.edu}

\author{Eugene Gorsky}
\address{Department of Mathematics\\ University of California, Davis\\ One Shields Avenue, Davis, CA, USA}
\email{egorskiy@math.ucdavis.edu}

\author{Jos\'e Simental}
\address{Instituto de Matem\'aticas, Universidad Nacional Aut\'onoma de M\'exico. Ciudad Universitaria, CDMX,  M\'exico}
\email{simental@im.unam.mx}

\title{The elliptic Hall algebra and the double Dyck path algebra}

\begin{document}

\begin{abstract}
    We show that the positive half $\mathcal{E}_{q,t}^{>}$ of the elliptic Hall algebra is embedded as a natural spherical subalgebra inside the double Dyck path algebra $\mathbb{B}_{q,t}$ introduced by Carlsson, Mellit and the second author. For this, we use the Ding-Iohara-Miki presentation of the elliptic Hall algebra and identify the generators inside $\mathbb{B}_{q,t}$. In order to obtain the entire elliptic Hall algebra $\mathcal{E}_{q,t}$, we define a \lq\lq double'' $\mathbb{DB}_{q,t}$ of the double Dyck path algebra, together with its positive and negative subalgebras and an involution that exchanges them.  
\end{abstract}

\maketitle


\section{Introduction}

In this paper, we study the relation between the \emph{elliptic Hall algebra} $\EHA$ and the \emph{double Dyck path algebra} $\Bqt$. 


The \newword{elliptic Hall algebra} was introduced by Burban and Schiffmann in \cite{EHA} and in the last decade found remarkable applications in algebraic geometry \cite{negut2015moduli,negut2018hecke,neguț2019surfaces,KivTsai,zhao2021k,zhao2019feigin}, mathematical physics \cite{neguct2018q,neguct2022toward,neguct2016exts}, combinatorics \cite{BGLX,BHMPS,BHMPS2,BHMPS3,Mellit}, knot theory \cite{GN,MortonSamuelson} and categorification \cite{GNtrace1,GNtrace2}. 

Although $\EHA$ has many equivalent definitions, in this paper we will use the isomorphism \cite{SchiffmannDrinfeld} between $\EHA$ and the so-called {\bf Ding-Iohara-Miki algebra} \cite{FFJMM}. This allows us to define  $\EHA$ by generators $e_{m},f_m,\ m\in \Z$ and $\psi_m^{\pm},\ m\in \Z_{\ge 0}$ modulo certain explicit relations, see Section \ref{sec:EHA-generators-and-relations}. The algebra $\EHA$ is graded with all $e_m$ of degree 1, all $f_m$ of degree $(-1)$ and all $\psi_m^{\pm}$ of degree zero. In particular, one can study the so-called \emph{positive} subalgebras:
$$
\EHA^{>}=\langle e_m\rangle \subset \EHA^{\geq}=\langle \psi_{m}^{\pm}, e_m\rangle\subset \EHA.
$$
The algebra $\EHA$ can be reconstructed from $\EHA^{\geq}$ by a formal Drinfeld double construction \cite{EHA} and admits an anti-involution $\Theta_{\mathcal{E}}$ exchanging the algebra $\EHA^{\geq}$ with the \emph{negative} part $\EHA^{\leq} = \langle \psi_m^{\pm}, f_m\rangle$. 

The algebra $\EHA$ acts on the space $\Sym_{q,t}$ of symmetric functions over $\C(q,t)$ in infinitely many variables. This polynomial representation was defined and studied by Feigin-Tsymbaliuk \cite{FT} and Schiffmann-Vasserot \cite{SV}; see Section \ref{sec: EHA polynomial} for the explicit formulas. In particular, the generators $e_m$ increase the degree of a symmetric function by one, the generators $f_m$ decrease the degree by one, and the generators $\psi^{\pm}_m$ preserve the degree.
\smallskip

The \newword{double Dyck path algebra} $\Bqt$ (or rather its cousin $\Aqt$) was introduced by Carlsson and Mellit \cite{Shuffle} in their proof of celebrated Shuffle Conjecture in algebraic combinatorics, further  generalized in \cite{Mellit}. Later work found deep connections between $\Bqt$, algebraic geometry of  parabolic Hilbert schemes \cite{CGM}, representation theory of double affine Hecke algebras (DAHA) \cite{milo,Milo2,Milo3,IonWu}, and skein theory \cite{Novarini, GH}. A characteristic feature of $\Bqt$ is a collection of idempotents $\e_k, k\ge 0$ which decompose any\footnote{We will only consider representations $V$ of $\Bqt$ which are locally finite in the sense that for any $v \in V$, $\e_kv \neq 0$ only for finitely many $k$.} 
representation $V$ as a direct sum 
\begin{equation}
\label{eq: V decomposition intro}
V=\bigoplus_{k=0}^{\infty}V_k,\ V_k=\e_kV.
\end{equation}
The generators of $\Bqt$ act between the different subspaces $V_k$, namely for each $k$ there is:
\begin{itemize}
\item an operator $d_{+}:V_k\to V_{k+1}$,
\item an operator $d_{-}:V_k\to V_{k-1}$,
\item an affine Hecke algebra $\AH_k$ with generators $T_i,z_i$ acting on the space $V_k$.
\end{itemize}
The algebra $\Bqt$ also has a polynomial representation \eqref{eq: V decomposition intro} such that $V_0=\Sym_{q,t}$ and
$V_k=\Sym_{q,t}[y_1,\ldots,y_k]$ for $k>0$. 
In our previous work \cite{calibrated} we constructed a large class of \newword{calibrated} representations of $\Bqt$ that, in particular, includes the polynomial representation.

It was observed in \cite{CGM} and  \cite{Mellit} that the polynomial representations of $\EHA^{>}$ and $\e_0\Bqt\e_0$ are similar (see Lemma \ref{lem: polynomial reps agree} for a more precise statement). 

\begin{proposition}
Define the elements
\begin{equation}
\label{eq: def e intro}
e_m:=\e_0d_{-}z_1^md_+\e_0\in \e_0 \Bqt \e_0,\ m\in \Z.
\end{equation}
Then the action of $e_m$ in the polynomial representation agrees with the action of the eponymic generators in $\EHA^{>}$, up to a scalar.
\end{proposition}

The main result of the paper (Theorem \ref{thm: main intro} below)  
greatly strengthens this result, by showing that in fact the \emph{spherical subalgebra} $\e_0 \Bqt \e_0$ of $\Bqt$ is generated by the elements $e_m$ and isomorphic to $\EHA^{>}$. 

\subsection{The Double of $\Bqt$}

The algebra $\Bqt$ has a nonnegative grading (with $\deg(d_-)=\deg(T_i)=\deg(z_i)=0$ and $\deg(d_+)=1$) which agrees with the grading on the polynomial representation. In particular, any element of $\e_0\Bqt\e_0$ acting on $\Sym_{q,t}$ must increase the degree, and there is no hope of obtaining the generators $f_m\in \EHA$ using $\Bqt$.

To circumvent this problem, we define a new algebra $\DBqt$, a ``double" of $\Bqt$ akin to the Drinfeld double construction of $\EHA$ in which we define the new families of elements $\psi_m^\pm$ and
\begin{equation}
\label{eq: def f intro}
f_m:=\e_0d_{+}z_1^md_{-}\e_0\in \e_0 \DBqt \e_0,\ m\in \Z.
\end{equation}
$\DBqt$ also has ``positive" and ``negative" parts exchanged by an anti-involution $\adj$, and a
family of new idempotents $\e_{-k}=\Theta(\e_k),k\ge 0$. The
decomposition in \eqref{eq: V decomposition intro} transforms into
\begin{equation}
\label{eq: V decomposition intro double}
V=\bigoplus_{k=-\infty}^{\infty}V_k,\ V_k=\e_kV.
\end{equation}
Our first main result defines the polynomial representation of $\DBqt$ and compares it with the polynomial representation of $\EHA$.

\begin{theorem}
\label{thm: polynomial double intro}
The algebra $\DBqt$ admits a polynomial representation $V=\bigoplus_{k=-\infty}^{\infty}V_k$ with 
$$
V_k=\begin{cases}
\Sym_{q,t}[y_1,\ldots,y_k] & \text{if}\ k>0\\
\Sym_{q,t} & \text{if}\ k=0\\
\Sym_{q,t}[y_1,\ldots,y_{|k|}] & \text{if}\ k<0.\\
\end{cases}
$$
In particular, the action of the generators $e_m$, $f_m$, and $\psi_m^\pm$ on $V_0=\Sym_{q,t}$ defined in \eqref{eq: def e intro} and \eqref{eq: def f intro} agrees with the action of the eponymic generators of $\EHA$, up to scalars.
\end{theorem}

\begin{remark}
More precisely, for $k > 0$ the space $V_k = \Sym_{q,t}[y_1, \dots, y_k] = \Sym_{q,t} \otimes \C[y_1, \dots, y_k]$ has a natural grading, and $V_{-k}$ is the graded dual space. Since the polynomial representation of $\Bqt$ is calibrated, the space $V_k$ comes equipped with a natural graded basis \cite{OrrMilo} and this allows us to identify $V_k^{\vee}$ with $V_k$, see Sections \ref{sec:polynomial-bqt} and \ref{sec:polynomial-dbqt} for details. 
\end{remark}

\subsection{The Spherical Subalgebra}

Next, we study the {\bf spherical subalgebra} $\e_0\DBqt\e_0$ in detail. The following is the primary result of the paper.

\begin{theorem}
\label{thm: main intro}
There exists an isomorphism $\e_0\DBqt\e_0\simeq \EHA$ which sends the generators $e_m,f_m,$ and $\psi_m^\pm$ to their namesakes in $\EHA$.
\end{theorem}

An immediate consequence of the polynomial representation of $\EHA$ being faithful \cite[Proposition 4.10]{SV}, is the following.

\begin{corollary}
The polynomial representation of $\e_0\DBqt\e_0$ (and hence of $\e_0\Bqt\e_0$) is faithful.
\end{corollary}
It remains unknown whether the polynomial representation of $\DBqt$, or even of $\Bqt$, is faithful. 

The proof of Theorem \ref{thm: main intro} consists of several different steps: 
\smallskip

In the {\bf first step}, we introduce the subalgebra $\Eqt\subset \e_0\DBqt\e_0$ generated by $e_m,f_m,$ and $\psi_m^\pm$ and prove the following:

\begin{theorem}
\label{thm: intro step 1}
We have the isomorphism $\Eqt\simeq \EHA$ which sends the generators $e_m,f_m$ and $\psi_m^\pm$ from \eqref{eq: def e intro},\eqref{eq: def f intro} to their eponyms in $\EHA$.
\end{theorem}

The proof of Theorem \ref{thm: intro step 1} occupies Section \ref{sec: EHA relations}. We first use the relations in $\Bqt$ to prove that the elements $e_m\in \e_0\Bqt\e_0$ satisfy the defining relations \eqref{eq: quadratic e} and \eqref{eq: cubic e} in $\EHA^{>}$. Next, we use the involution $\Theta$ to prove the analogous result for $f_m$. The rest of the relations involving the generators  $\psi_m^{\pm}$ follow more or less automatically (in particular, the equation for the commutator $[e_m,f_k]$ is built into the definition of $\DBqt$). Altogether, this yields a surjective homomorphism
$$
\EHA\twoheadrightarrow \Eqt.
$$
Using Theorem \ref{thm: polynomial double intro} and faithfulness of the
polynomial representation of $\EHA$, we conclude that this is an isomorphism.

For the {\bf second step}, we define certain elements
\begin{equation}
\label{eq: def Y intro}
Y_{m_1,\ldots,m_n}:=\e_0d_{-}z_1^{m_1}\varphi z_2^{m_2}\cdots \varphi z_1^{m_n}d_{+}\e_0\in \e_0\Bqt\e_0\quad m_i\in \Z,
\end{equation}
where $\varphi:=\frac{1}{q-1}[d_{+},d_{-}]$.
The elements $Y_{m_1,\ldots,m_n}$ were previously considered (in the polynomial representation of $\Bqt$) in \cite[Section 8.2]{CGM}.
Evidently, for $n=1$ we have
$Y_{m}=e_m \in \Eqt$. However, it is not clear from this definition that $Y_{m_1,\ldots,m_n}\in \Eqt$  for $n>1$. Nevertheless, we prove the following.

\begin{theorem}
\label{thm: intro step 2}
The elements $Y_{m_1,\ldots,m_n}$ generate the spherical subalgebra $\e_0\Bqt\e_0$.
\end{theorem}

In Section \ref{sec: special} we describe an explicit inductive procedure for writing an arbitrary element of $\e_0\Bqt\e_0$ as a polynomial in $Y_{m_1,\ldots,m_n}$. In fact, we define a ``higher level" analogues of $Y_{m_1,\ldots,m_n}$  in $\e_0\Bqt\e_k$ for $k\ge 0$ and prove that they generate $\e_0\Bqt\e_k$, see Theorem \ref{thm: special generate}.

Finally, in the {\bf third step} we prove the containment $Y_{m_1,\ldots,m_n}\in \Eqt$ for all $n$. 
Presenting $Y_{m_1,\ldots,m_n}$ as explicit polynomials in $e_k$ seems to be a 
challenging problem 
so we choose a rather indirect approach inspired by \cite{negut2018hecke,GNtrace1}. First, we prove some commutation relations for $Y_{m_1,\ldots,m_n}$.

\begin{theorem}
\label{thm: intro step 3}
The following relations hold:
\begin{equation}
\label{eq: Y relation 1 intro}
Y_{m_1,\ldots,m_i,m_{i+1},\ldots,m_n}-qtY_{m_1,\ldots,m_i-1,m_{i+1}+1,\ldots,m_n}=-Y_{m_1,\ldots,m_i}Y_{m_{i+1},\ldots,m_n},
\end{equation}
\begin{equation}
\label{eq: Y relation 2 intro}
[e_k,Y_{m_1,\ldots,m_n}]=(t-1)(q-1)\sum_{i=1}^{n}\begin{cases}
\sum_{a=1}^{k-m_i}Y_{m_1,\ldots,m_{i-1},k-a,m_i+a,m_{i+1},\ldots,m_n} & \text{if}\ k>m_i\\
0 & \text{if}\ k=m_i\\
-\sum_{a=1}^{m_i-k}Y_{m_1,\ldots,m_{i-1},m_i-a,k+a,m_{i+1},\ldots,m_n} & \text{if}\ k< m_i.
\end{cases}
\end{equation}
\end{theorem}

\begin{remark}
Theorem \ref{thm: intro step 3} is motivated by the work of Negu\cb{t} who constructed analogous elements $X_{m_1,\ldots,m_n}\in \EHA^{>}$ in \cite[Proposition 6.2]{NegutShuffle} and proved in \cite{negut2018hecke} that they satisfy the relations \eqref{eq: Y relation 1 intro} and \eqref{eq: Y relation 2 intro}.
See also \cite[Proposition 4.2.1]{BHMPS}.
\end{remark}

The proof of these relations occupies Section \ref{sec: Y proofs}. Next, we use the results of \cite{GNtrace1} and equations \eqref{eq: Y relation 1 intro}-\eqref{eq: Y relation 2 intro} to prove  in Theorem \ref{thm: spherical generated by e} that $Y_{m_1,\ldots,m_n}$ can be expressed as  polynomials in $e_m$.

\begin{corollary}
\label{cor: intro Y in E}
The elements $Y_{m_1,\ldots,m_n}$ are contained in $\Eqt$.
\end{corollary}

By combining Theorem \ref{thm: intro step 2} and Corollary \ref{cor: intro Y in E} we conclude that $\e_0\Bqt\e_0$ is indeed generated by the elements $e_m$. This implies:

\begin{theorem}
\label{thm: intro spherical is E}
We have $\e_0\DBqt\e_0=\Eqt$.
\end{theorem}

By combining Theorems \ref{thm: intro step 1} and \ref{thm: intro spherical is E}, we obtain Theorem \ref{thm: main intro} proving that $\e_0\DBqt\e_0\simeq \EHA$.

\subsection{Future Directions}

Theorem \ref{thm: main intro} opens up several new directions and questions in the representation theory of both $\EHA$ and $\DBqt$. Some of these include:

\begin{enumerate}
\item Which calibrated representations of $\Bqt$ from \cite{calibrated} can be lifted to representations of $\DBqt$? By Theorem \ref{thm: main intro} these would yield representations of $\EHA$. We plan to address this question in future work \cite{caldouble}.
\item Given a representation of $\EHA$, is it possible to lift it to a representation of $\DBqt$ or $\Bqt$? 
\item A product of several $d_{+}$ and $d_{-}$ in $\Bqt$ corresponding to a Dyck path defines an element of $\e_0\Bqt\e_0$. The action of such elements in the polynomial representation played a crucial role in \cite{Shuffle}. Is it possible to define the corresponding element of $\EHA$ more explicitly? What if we add monomials in $z_i$ between $d_{\pm}$?
\item Blasiak et al. \cite{BHMPS3} recently defined lots of interesting elements of $\EHA$ generalizing $Y_{m_1,\ldots,m_n}$. Is it possible to write these elements in $\Bqt$ more explicitly?
\item The elliptic Hall algebra $\EHA$ is part of a large family of {\bf quantum toroidal algebras} which can be assigned to an arbitrary affine Dynkin diagram ($\EHA$ corresponds to $\widehat{A_1}$). Is it possible to define the analogues of $\Bqt$ or $\DBqt$ for arbitrary Dynkin diagrams?
\item Is the polynomial representation of $\DBqt$ faithful? 
\item The elliptic Hall algebra has an integral form defined in \cite{neguct2022integral}, it would be very interesting to define a matching integral form of $\Bqt$.  Note that Theorem \ref{thm: main intro} requires working over $\C(q,t)$, since expressing the elements $Y_{m_1,\ldots,m_n}$ in terms of $e_k$ requires denominators. 
See Example \ref{eq: Y 00} for a sample  of this phenomenon.
\item There is an $\mathrm{SL}_2(\Z)$-action on the elliptic Hall algebra $\EHA$. Does this action extend to the algebra $\DBqt$? Some indication of this is given in \cite[Section 3]{Mellit}.
\end{enumerate}

\subsection*{Acknowledgments}

The authors would like to thank Milo Bechtloff Weising, Erik Carlsson, Anton Mellit, Andrei Negu\cb{t} and Monica Vazirani for many helpful discussions. The work of E. G. was partially supported by the NSF grant DMS-2302305. The work of J. S. was partially supported by CONAHCyT project CF-2023-G-106 and UNAM’s PAPIIT Grant IA102124.

\section{The Elliptic Hall Algebra}

\subsection{Generators and Relations}\label{sec:EHA-generators-and-relations} As mentioned in the introduction, the \newword{elliptic Hall algebra} arises under many guises and names in the literature. Here, we follow \cite{FFJMM} and consider the algebra $\EHA$ with generators $e_m,f_m\ (m\in \Z)$ and $\psi^+_{m},\psi^{-}_{-m}\ (m\in \Z_{\ge 0}).$  It is convenient to package these in various generating series 
\[
e(\z)=\sum_{m\in \Z}e_m\z^{-m},\ f(\z)=\sum_{m\in \Z}f_m\z^{-m},\ \psi^{+}(\z)=\sum_{m\in \Z_{\ge 0}}\psi^{+}_m\z^{-m},\ \psi^{-}(\z)=\sum_{m\in \Z_{\geq 0}}\psi^{-}_{-m}\z^{m}.
\]

Now, fixing parameters $q_1=q$, $q_2=t$ and $q_3=q^{-1}t^{-1}$ so that $q_1q_2q_3=1$, and letting $\sigma_1$ and $\sigma_2$ denote the elementary symmetric functions:
\begin{align*}
\sigma_1&=q_1+q_2+q_3 = q+t+q^{-1}t^{-1},\\
\sigma_2&=q_1q_2+q_2q_3+q_1q_3 = qt+q^{-1}+t^{-1},
\end{align*}
define the polynomial
\begin{equation}
\label{eq: def g}
g(\z,\w):=(\z-q_1\w)(\z-q_2\w)(\z-q_3\w)=\z^3-\sigma_1\z^2\w+\sigma_2\z\w^2-\w^3.
\end{equation}
The defining relations of $\EHA$ are then given by \cite{FFJMM}:
\begin{subequations}
\begin{align}
\label{eq: quadratic e}
g(\z,\w)e(\z)e(\w)&=-g(\w,\z)e(\w)e(\z),
\\
\label{eq: quadratic f}
g(\w,\z)f(\z)f(\w)&=-g(\z,\w)f(\w)f(\z),
\\
\label{eq: commutation e with psi}
g(\z,\w)\psi^{\pm}(\z) e(\w)&=-g(\w,\z)e(\w)\psi^{\pm}(\z),\ 
\\
\label{eq: commutation f with psi}
g(\w,\z)\psi^{\pm}(\z) f(\w)&=-g(\z,\w)f(\w)\psi^{\pm}(\z),
\end{align}
\begin{equation}
\label{eq: commutation of e and f}
[e(\z),f(\w)]=g(1,1)^{-1}\left(\sum_{n\in \Z}\z^n \w^{-n}\right)(\psi^+(\w)-\psi^{-}(\z)),
\end{equation}
\begin{equation}
\label{eq: psi commute}
[\psi_i^{\pm},\psi_{j}^{\pm}]=0,\ [\psi_i^{\pm},\psi_{j}^{\mp}]=0,
\end{equation}
\begin{equation}
\label{eq: psi 0 invertble}
\psi_0^{\pm}(\psi_0^{\pm})^{-1}=(\psi_0^{\pm})^{-1}\psi_0^{\pm}=1,
\end{equation}
\begin{align}
\label{eq: cubic e}
[e_0,[e_1,e_{-1}]]&=0,\\
\label{eq: cubic f}
[f_0,[f_1,f_{-1}]]&=0.
\end{align}
\end{subequations}

Notice that \eqref{eq: psi commute} simply says that $\psi_i^{\pm}$ pairwise commute, and \eqref{eq: psi 0 invertble} that both $\psi_0^{+}$ and $\psi_0^{-}$ are invertible. 

A less trivial consequence of \eqref{eq: commutation e with psi} and \eqref{eq: commutation f with psi} is that both $\psi_0^{+}$ and $\psi_0^{-}$ commute with all $e_m$ and $f_m$ and hence are central in $\EHA$. Indeed, looking at the coefficient of $\z^{3}\w^{-m}$ (for $\psi^+$) and $\z^0 \w^{-m}$ (for $\psi^-$)
on both sides of  \eqref{eq: commutation e with psi} we obtain that $\psi_0^{\pm}e_m = e_m\psi_0^{\pm}$, and similarly that $\psi_0^{\pm}f_m = f_m\psi_0^{\pm}$. Note, it is important here that the generating series $\psi^{+}(\z)$ and $\psi^{-}(\z)$ are infinite only on one side. 

Let us also elaborate on the relation \eqref{eq: commutation of e and f}. By expanding the formal power series on both sides, we see that \eqref{eq: commutation of e and f} is equivalent to:
\begin{equation}\label{eq:explicit-commutation-e-and-f}
    [e_n, f_m] = \begin{cases} g(1,1)^{-1}\psi_{m+n}^{+}, & \text{if} \; m+n > 0 \\ -g(1,1)^{-1}\psi_{-(m+n)}^{-} & \text{if} \; m+n < 0, \\
    g(1,1)^{-1}(\psi_0^{+} - \psi_0^{-}) & \text{if} \; m+n = 0
    \end{cases} 
\end{equation}

The following is clear from the relations.
\begin{lemma}
\label{lem: theta e f EHA}
The algebra $\EHA$ has a $\C(q,t)$-linear 
anti-involution 
$\Theta_{\mathcal{E}}$ defined by 
$$
\Theta_{\mathcal{E}}(e_m)=f_m,\ \Theta_{\mathcal{E}}(f_m)=e_m,\ \Theta_{\mathcal{E}}(\psi^{\pm}_m)=\psi^{\pm}_m.
$$
\end{lemma}

We will   work with various subalgebras of $\EHA$.

\begin{definition}\label{def:pos EHA}
The \emph{positive subalgebra} $\EHA^{\geq}$ has generators $e_m, m\in \Z$ and $\psi^+_m,\psi^{-}_{-m}$ ($m\ge 0$) satisfying the relations 
\eqref{eq: quadratic e}, \eqref{eq: commutation e with psi}, \eqref{eq: psi commute}, \eqref{eq: psi 0 invertble}, and \eqref{eq: cubic e}.
\end{definition}

\begin{definition}\label{def:neg EHA}
The \emph{negative subalgebra} $\EHA^{\leq}$ has generators $f_m, m\in \Z$ and $\psi^+_m,\psi^{-}_{-m}$ ($m\ge 0$) satisfying the relations \eqref{eq: quadratic f}, \eqref{eq: commutation f with psi}, \eqref{eq: psi commute}, \eqref{eq: psi 0 invertble}, and \eqref{eq: cubic f}.
\end{definition}

We similarly define the subalgebra 
$\EHA^{>}\subset \EHA^{\geq}$ (resp. $\EHA^{<}\subset \EHA^{\leq}$) with generators $e_m, m\in \Z$ (resp. $f_m, m\in \Z$)  modulo the relations \eqref{eq: quadratic e} and \eqref{eq: cubic e} (resp. \eqref{eq: quadratic f} and \eqref{eq: cubic f}).

Evidently, the involution $\Theta_{\mathcal{E}}$ exchanges $\EHA^{>}$ with $\EHA^{<}$, and $\EHA^{\geq}$ with $\EHA^{\leq}$.

\subsection{The Polynomial Representation of $\EHA$}
\label{sec: EHA polynomial} Following Feigin and Tsymbaliuk \cite{FT}, we now describe the polynomial representation of $\EHA$ on the space $\Sym_{q,t}$ of symmetric functions over $\C(q,t)$. 

 Given any partition $\lambda$ and a box $\Box \in \lambda$ in row $r$ and column $c$, the \newword{(q,t)-content} of this box is the scalar $q^{c-1}t^{r-1}$. In what follows we abuse notation, and denote both the box in the partition and its associated $(q,t)$-content by $\Box$. Since the position of a box is uniquely determined by its $(q,t)$-content, we also sometimes abuse notation and write $\lambda \cup x$ to the denote the partition obtained from $\lambda$ by adding the box whose $(q,t)$-content is equal to $x$.

So then, taking the basis of $\Sym_{q,t}$ to be the modified Macdonald polynomials $I_\lambda := \widetilde{H}_\lambda(X;q,t)$, the action is given by \cite[Lem. 3.1 and Prop. 3.1]{FT}:
\begin{subequations}
\begin{align}
\label{eq: feigin tsymbaliuk e}
e_mI_{\lambda}& =\sum_{\mu=\lambda\cup x} \frac{(q^{\mu_i-1}t^{i-1})^m}{1-qt}\prod_{j=1}^{\infty}\frac{(1-q^{\mu_j-\mu_i+1}t^{j-i+1})}{(1-q^{\mu_{j}-\mu_i+1}t^{j-i})}I_{\mu},
\\
\label{eq: feigin tsymbaliuk f}
f_mI_{\mu}& =\sum_{\lambda=\mu-x}
\frac{(q^{\lambda_i}t^{i-1})^{m-1}}{1-qt}\prod_{j=1}^{\infty}\frac{1-q^{\lambda_i-\lambda_{j}+1}t^{i-j+1}}{1-q^{\lambda_{i}-\lambda_j+1}t^{i-j}}I_{\lambda},
\end{align}
where the sums are taken over all partitions $\mu$ (resp. $\lambda$) that can be obtained from $\lambda$ (resp. $\mu$) by adding (resp. removing) a box, with $i$ equal to the row of the box.

 We also have, 
\begin{align}
\label{eq: feigin tsymbaliuk psi}
\psi^{\pm}(\z)I_{\lambda}& = \left(-\frac{1-q^{-1}t^{-1}\z^{-1}}{1-\z^{-1}}\prod_{\square\in \lambda}\frac{(1-q^{-1}\square \z^{-1})(1-t^{-1}\square \z^{-1})(1-qt\square \z^{-1})}{(1-q\square \z^{-1})(1-t\square \z^{-1})(1-q^{-1}t^{-1}\square \z^{-1})} \right)I_{\lambda},
\end{align}
\end{subequations}
where the coefficient of $I_{\lambda}$ in the RHS of \eqref{eq: feigin tsymbaliuk psi} is a rational function in $\z$, and should thus be expanded as a power series in $\z^{\mp 1}$ to obtain the formula for the action of $\psi^{\pm}(\z)$.

\begin{theorem}[\cite{SV}, Proposition 4.10]
\label{thm: poly EHA faithful}
The polynomial representation of $\EHA$ is faithful.
\end{theorem}

It will be useful to slightly rewrite equations \eqref{eq: feigin tsymbaliuk e}-\eqref{eq: feigin tsymbaliuk psi}. 
Define the scalars in $\C(q,t)$,
\begin{subequations}
\begin{align}
\label{eq: c for partitions}
c(\lambda;x)&:=-\Lambda\left(-x^{-1}+(1-q)(1-t)B_{\lambda}x^{-1}+1\right)\\
\label{eq: c* for partitions}
c^*(\mu;x)&:=x^{-1}\Lambda\big(-(1-q)(1-t)B^*_{\lambda}x\big)
\end{align}  
\end{subequations}
where $\lambda, \mu$ are partitions with $\mu = \lambda \cup x$, and  
$$
B_{\lambda}=\sum_{\sq\in \lambda}\sq,\quad B^*_{\lambda}=\sum_{\sq\in \lambda}\sq^{-1}, \quad \text{and} \quad \Lambda\left(\sum \phi_{i,j}q^it^j\right)=\prod (1-q^it^j)^{\phi_{i,j}}.
$$

\begin{lemma}
\label{lem: c for partitions as a product}
Suppose that $\mu=\lambda\cup \Box$ with $\Box$ in row $i$ and $(q,t)$-content of $\Box$ is equal to $x$. 
Then the coefficients \eqref{eq: c for partitions} and \eqref{eq: c* for partitions} can be written as
$$
c(\lambda;x)=-(1-t)\prod_{j\neq i}\frac{(1-q^{\lambda_j-\lambda_i}t^{j-i+1})}{(1-q^{\lambda_j-\lambda_i}t^{j-i})}=-(1-t)\prod_{j\neq i}\frac{(1-q^{\mu_j-\mu_i+1}t^{j-i+1})}{(1-q^{\mu_j-\mu_i+1}t^{j-i})}.
$$
and
$$
c^*(\mu;x)=\frac{(q^{\lambda_i}t^{i-1})^{-1}}{(1-q)}\prod_{j\neq i}\frac{(1-q^{\lambda_i-\lambda_j+1}t^{i-j+1})}{(1-q^{\lambda_i-\lambda_j+1}t^{i-j})}
$$
\end{lemma}

\begin{proof}
We expand $B_{\lambda}$ and $B^*_{\lambda}$ as  sums of geometric sequences:
$$
B_{\lambda}=\sum_{j=1}^{\infty}(1+\ldots+q^{\lambda_j-1})t^{j-1}=\sum_{j=1}^{\infty}\frac{1-q^{\lambda_j}}{1-q}t^{j-1}=\frac{1}{(1-q)(1-t)}-\frac{1}{1-q}\sum_{j=1}^{\infty}q^{\lambda_j}t^{j-1}
$$
and similarly
$$
B^*_{\lambda}=\frac{1}{(1-q^{-1})(1-t^{-1})}-\frac{1}{1-q^{-1}}\sum_{j=1}^{\infty}q^{-\lambda_j}t^{1-j}.
$$
Note that $x=q^{\lambda_i}t^{i-1}$. Therefore
\begin{align*}
-x^{-1}+(1-q)(1-t)B_{\lambda}x^{-1}+1&=-(1-t)\sum_{j=1}^{\infty}q^{\lambda_j}t^{j-1}x^{-1}+1
\\
&=-\sum_{j\neq i}q^{\lambda_j-\lambda_i}(t^{j-i}-t^{j-i+1})-(1-t)+1\\
&=-\sum_{j\neq i}q^{\lambda_j-\lambda_i}(t^{j-i}-t^{j-i+1})+t.
\end{align*}
and
$$
c(\lambda;x)=-(1-t)\prod_{j\neq i}\frac{(1-q^{\lambda_j-\lambda_i}t^{j-i+1})}{(1-q^{\lambda_j-\lambda_i}t^{j-i})}.
$$
For the first equation it suffices to notice that $\mu_i=\lambda_i+1$ and $\mu_j=\lambda_j$ for $j\neq i$. For the second equation, we write
\begin{align*}
-(1-q)(1-t)B^*_{\lambda}x&=-qt-(1-t)q\sum_{j=1}^{\infty}q^{\lambda_i-\lambda_j}t^{i-j}
\\
&=-q-\sum_{j\neq i}q^{\lambda_i-\lambda_j+1}(t^{i-j}-t^{i-j+1}).
\end{align*}
\end{proof}

\begin{corollary}
\label{cor: e and f short}
We have
$$
e_mI_{\lambda}=-\frac{1}{(1-q)(1-t)}\sum_{\mu=\lambda\cup x}c(\lambda;x)x^mI_{\mu},\quad 
f_mI_{\mu}=\sum_{\lambda=\mu-x}c^*(\mu;x)x^m I_{\lambda}
$$
\end{corollary}

\begin{lemma}
\label{lem: c c*}
Let 
\begin{equation}
\label{eq: def d lambda}
d_{\lambda}=\left(\prod_{\sq \in \lambda}\sq\right)
\cdot \Lambda\big(-B_{\lambda}^*+(1-q)(1-t)B_{\lambda}B^*_{\lambda}\big).
\end{equation}
Then for $\mu=\lambda\cup x$ we have
$$
\frac{c(\lambda;x)}{c^*(\mu;x)}=-\frac{(1-q)(1-t)d_{\mu}}{(1-qt)d_{\lambda}}.
$$
\end{lemma}

\begin{proof}
Indeed, $B_{\mu}=B_{\lambda}+x,B^*_{\mu}=B^*_{\lambda}+x^{-1}$, so
\begin{align*}
\frac{d_{\mu}}{d_{\lambda}}&=x\Lambda\big(-x^{-1}+(1-q)(1-t)B_{\lambda}x^{-1}+(1-q)(1-t)B^*_{\lambda}x+(1-q)(1-t)\big)
\\
&=-\frac{(1-qt)c(\lambda;x)}{(1-q)(1-t)c^*(\mu;x)}.
\end{align*}
\end{proof}

\begin{lemma}
\label{lem: exp identity}
We have the power series identities
$$
\exp\left[-\sum_{m=1}^{\infty}\frac{\z^{-m}}{m}\w^m(1-q^m)(1-t^m)(1-(qt)^{-m})\right]=-\frac{g(\w,\z)}{g(\z,\w)}
$$
and 
$$
\exp\left[-\sum_{m=1}^{\infty}\frac{\z^{m}}{m}\w^{-m}(1-q^m)(1-t^m)(1-(qt)^{-m})\right]=-\frac{g(\z,\w)}{g(\w,\z)}.
$$
\end{lemma}
\begin{proof}
We prove the first equation, the second follows after swapping $\z$ and $\w$. Note that 
$$
(1-q^m)(1-t^m)(1-(qt)^{-m})=-(q^m+t^m+(qt)^{-m}-q^{-m}-t^{-m}-(qt)^m),
$$
so
\begin{align*}
&\exp\left[-\sum_{m=1}^{\infty}\frac{\z^{-m}}{m}\w^m(1-q^m)(1-t^m)(1-(qt)^{-m})\right]
\\
&=\exp\left[\sum_{m=1}^{\infty}\frac{\z^{-m}}{m}\w^m(q^m+t^m+(qt)^{-m}-q^{-m}-t^{-m}-(qt)^m)\right]
\\
&=\frac{(1-q^{-1}\w\z^{-1})(1-t^{-1}\w\z^{-1})(1-(qt)\w\z^{-1})}{(1-q\w\z^{-1})(1-t\w\z^{-1})(1-(qt)^{-1}\w\z^{-1})}\\
&=-\frac{(\w-q\z)(\w-t\z)(\w-(qt)^{-1}\z)}{(\z-q\w\z)(\z-t\w\z)(\z-(qt)^{-1}\w)}=\frac{g(\w,\z)}{g(\z,\w)}.
\end{align*}
\end{proof}

\begin{corollary}
\label{cor: feigin tsymbaliuk psi exponential}
Equation \eqref{eq: feigin tsymbaliuk psi} can equivalently be written as follows:
$$
\psi^{\pm}(\z)I_{\lambda}=\left(-\frac{1-q^{-1}t^{-1}\z^{-1}}{1-\z^{-1}}\right)\prod_{\square\in \lambda}\exp\left[\mp\sum_{m=1}^{\infty}\frac{\z^{\mp m}}{m}\square^{\pm m}(1-q^m)(1-t^m)(1-(qt)^{-m})\right].
$$
\end{corollary}

\subsection{Monodromy Conditions}
For later use, we examine the so-called monodromy conditions, introduced in \cite{calibrated}, on the coefficients $c(\lambda;x)$ and $c^{\ast}(\mu; x)$ from \eqref{eq: c for partitions} and \eqref{eq: c* for partitions}, respectively. Consider the Young poset $P$. A \emph{system of coefficients} $\{c(\lambda;x)\}$ on $P$ is an assignment of a value $c(\lambda;x) \in \C(q,t)^{\times}$ for every covering relation $\lambda \prec \lambda\cup x$ on $P$. We say that the system of coefficients $\{c(\lambda;x)\}$ satisfies the \newword{monodromy conditions} if, whenever we have chains $\lambda \prec \lambda\cup x \prec \lambda \cup x \cup y$ and $\lambda \prec \lambda \cup y \prec \lambda \cup y \cup x$ in $P$ we have (cf. \cite[Lemma 3.24]{calibrated}): 
\begin{equation}\label{eq:monodromy-conditions}
\frac{c(\lambda;x)c(\lambda\cup x; y)}{c(\lambda;y)c(\lambda\cup y;x)} = -\frac{(x-ty)(x-qy)(y-qtx)}{(y-tx)(y-qx)(x-qty)}.
\end{equation}

\begin{lemma}[\cite{calibrated}, Theorem 3.32]\label{lem:monodromy-conditions}
The system of coefficients $c(\lambda;x)$ defined by \eqref{eq: c for partitions} satisfies the monodromy conditions.
\end{lemma}

Now let us deal with the opposite Young poset $P^{\opp}$. A system of coefficients $\{c^{\ast}(\mu;x)\}$ is an assignment of a value $c^{\ast}(\mu; x)$ for every covering relation $\mu \prec \mu -x$ on $P^{\opp}$. We say that $\{c^{\ast}(\mu;x)\}$ satisfies the \newword{dual monodromy conditions} if, whenever we have chains $\mu \prec \mu - x \prec \mu - x - y$ and $\mu \prec \mu - y \prec \mu - y - x$ the following equation is satisfied:

\begin{equation}\label{eq:dual-monodromy-condition}
\frac{c^{\ast}(\mu;x)c^{\ast}(\mu - x; y)}{c^{\ast}(\mu, y)c^{\ast}(\mu - y; x)} = -\frac{(y - tx)(y - qx)(x - qty)}{(x - ty)(x - qy)(y-qtx)}.
\end{equation}

\begin{lemma}\label{lem:dual-monodromy-conditions}
The system of coefficients $c^{\ast}(\lambda)$ defined  by \eqref{eq: c* for partitions} satisfies the dual monodromy conditions.
\end{lemma}
\begin{proof}
    We use Lemmas \ref{lem: c c*} and \ref{lem:monodromy-conditions}. First, note that by Lemma \ref{lem: c c*} we can write
    \[
    c^{\ast}(\mu;x) = -\frac{(1-qt)d_{\mu - x}}{(1-q)(1-t)d_{\mu}}c(\mu - x; x). 
    \]
    So that we have
\[
\frac{c^{\ast}(\mu;x)c^{\ast}(\mu - x; y)}{c^{\ast}(\mu, y)c^{\ast}(\mu - y; x)} = \frac{\frac{d_{\mu -x}}{d_{\mu}}\frac{d_{\mu - x - y}}{d_{\mu - x}}c(\mu - x; x)c(\mu - x - y; y)}{\frac{d_{\mu -y}}{d_{\mu}}\frac{d_{\mu - x - y}}{d_{\mu - y}}c(\mu - y; y)c(\mu - x - y;x)}  
\]
and the result now follows since $\{c(\lambda;x)\}$ satisfies the monodromy conditions. 
\end{proof}

\section{The Double of $\Bqt$} 

As discussed in the introduction, in order to realize the complete elliptic Hall algebra within the double Dyck path algebra it is necessary to enlarge the latter algebra first. Motivated by a similar decomposition for $\EHA$ we will define a new doubled algebra whose ``positive half'' is the usual double Dyck path algebra, together with an involution that exchanges the positive and negative parts.

\subsection{Definition of $\DBqt$}
\label{sec: def DBqt}
Consider the following infinite quiver $Q_{\Z}$:
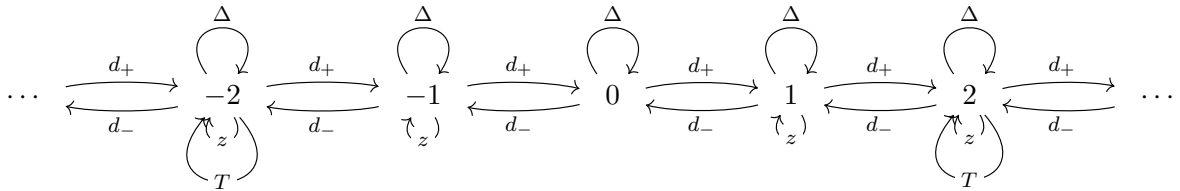
\[\begin{tikzcd}
	\cdots && {-2} && {-1} && 0 && 1 && 2 && \cdots
	\arrow["{d_{+}}", curve={height=-6pt}, shorten <=5pt, shorten >=5pt, from=1-1, to=1-3]
	\arrow["{d_-}", curve={height=-6pt}, shorten <=5pt, shorten >=5pt, from=1-3, to=1-1]
	\arrow["\Delta", from=1-3, to=1-3, loop, in=55, out=125, distance=10mm]
	\arrow["z"{description}, from=1-3, to=1-3, loop, in=240, out=300, distance=5mm]
	\arrow["T"{description}, from=1-3, to=1-3, loop, in=230, out=310, distance=15mm]
	\arrow["{d_{+}}", curve={height=-6pt}, shorten <=5pt, shorten >=5pt, from=1-3, to=1-5]
	\arrow["{d_-}", curve={height=-6pt}, shorten <=5pt, shorten >=5pt, from=1-5, to=1-3]
	\arrow["\Delta", from=1-5, to=1-5, loop, in=55, out=125, distance=10mm]
	\arrow["z"{description}, from=1-5, to=1-5, loop, in=240, out=300, distance=5mm]
	\arrow["{d_{+}}", curve={height=-6pt}, shorten <=5pt, shorten >=5pt, from=1-5, to=1-7]
	\arrow["{d_-}", curve={height=-6pt}, shorten <=5pt, shorten >=5pt, from=1-7, to=1-5]
	\arrow["\Delta", from=1-7, to=1-7, loop, in=55, out=125, distance=10mm]
	\arrow["{d_{+}}", curve={height=-6pt}, shorten <=5pt, shorten >=5pt, from=1-7, to=1-9]
	\arrow["{d_-}", curve={height=-6pt}, shorten <=5pt, shorten >=5pt, from=1-9, to=1-7]
	\arrow["\Delta", from=1-9, to=1-9, loop, in=55, out=125, distance=10mm]
	\arrow["z"{description}, from=1-9, to=1-9, loop, in=240, out=300, distance=5mm]
	\arrow["{d_{+}}", curve={height=-6pt}, shorten <=5pt, shorten >=5pt, from=1-9, to=1-11]
	\arrow["{d_-}", curve={height=-6pt}, shorten <=5pt, shorten >=5pt, from=1-11, to=1-9]
	\arrow["\Delta", from=1-11, to=1-11, loop, in=55, out=125, distance=10mm]
	\arrow["z"{description}, from=1-11, to=1-11, loop, in=240, out=300, distance=5mm]
	\arrow["T"{description}, from=1-11, to=1-11, loop, in=230, out=310, distance=15mm]
	\arrow["{d_{+}}", curve={height=-6pt}, shorten <=5pt, shorten >=5pt, from=1-11, to=1-13]
	\arrow["{d_-}", curve={height=-6pt}, shorten <=5pt, shorten >=5pt, from=1-13, to=1-11]
\end{tikzcd}\]
where:
\begin{itemize}
    \item Each loop with label $\Delta$ represents two infinite number of loops, labeled $\Delta_{p_m}$ and $\Delta^{\ast}_{p_m}$ for integers $m > 0$.
    \item The loop at a vertex $k$ labeled by $z$ represents $|k|$ loops, labeled $z_1, \dots, z_{|k|}$.
    \item The loop at vertex $k$ labeled by $T$ represents $|k|-1$ loops, labeled $T_1, \dots, T_{|k|-1}$.
    \item For a vertex $k$, we denote the lazy path at $k$ by $\e_k$. 
\end{itemize}

We consider the path algebra $\C(q,t)Q_\Z$ of $Q_\Z$ with coefficients in $\C(q,t)$, \textbf{assuming that each loop $z_i$ at $k$ is invertible in $\e_k\C(q,t)Q_\Z\e_k$.} Note that, since $Q_\Z$ has an infinite number of vertices, $\C(q,t)Q_\Z$ is non-unital. 

\begin{definition}
The algebra $\DBqt$ is the quotient of the path algebra $\C(q,t)Q_\Z$ by the following relations, where we set $\varphi := \frac{1}{q-1}[d_{+}, d_{-}]$.

\begin{enumerate}
\item[(DB$^{+}$)] The left action on $\e_k$ with $k>0$ satisfies:
\begin{subequations}
\begin{align}
(T_i - 1)(T_i + q) = 0, \qquad T_{i}T_{i+1}T_{i} &= T_{i+1}T_{i}T_{i+1} \qquad T_{i}T_{j} = T_{j}T_{i} \, (|i - j| > 1) \label{eq:hecke relns} \\
T_{i}^{-1}z_{i+1}T_{i}^{-1} &= q^{-1}z_{i} \, (1 \leq i \leq k-1) \label{eq:T and z} \\
z_{i}T_{j} = T_{j}z_{i} \, (i \not\in \{j, j+1\}),& \qquad z_{i}z_{j} = z_{j}z_{i} (1 \leq i, j \leq k) \label{eq:affine Hecke relns}  \\
d^{2}_{-}T_{k-1} = d^{2}_{-} \, (k \geq 2),& \qquad d_{-}T_{i} = T_{i}d_{-} \, (1 \leq i \leq k-2) \label{eq:T d-} \\
T_{1}d_{+}^{2} = d_{+}^{2},& \qquad d_{+}T_{i} = T_{i+1}d_{+} \, (1 \leq i \leq k-1) \label{eq:T d+} \\
q\varphi d_{-} = d_{-}\varphi T_{k-1} \, (k \geq 2),& \qquad T_{1}\varphi d_{+} = qd_{+}\varphi \label{eq:phi} \\
z_{i}d_{-} = d_{-}z_{i},& \qquad d_{+}z_{i} = z_{i+1}d_{+} \label{eq: d z} \\
z_{1}(qd_{+}d_{-} - d_{-}d_{+}) &= qt(d_{+}d_{-} - d_{-}d_{+})z_{k} \label{eq:qphi} \\
[\Delta_{p_m}, \Delta_{p_n}] = [\Delta_{p_m}, T_i] &=   [\Delta_{p_m}, z_j] = [\Delta_{p_m}, d_{-}] = 0 \label{eq:delta} \\
[\Delta^{\ast}_{p_m}, \Delta^{\ast}_{p_n}] = [\Delta^{\ast}_{p_m}, T_i] &=  [\Delta^{\ast}_{p_m}, z_j] =  [\Delta^{\ast}_{p_m}, d_{-}] = 0 \label{eq:delta star} \\
[\Delta^{\ast}_{p_m}, \Delta_{p_n}] & = 0 \label{eq: delta delta star} \\
[\Delta_{p_m}, d_{+}] = z_1^{m}d_{+}, & \qquad [\Delta^{\ast}_{p_m}, d_{+}] = z_1^{-m}d_{+}. \label{eq: delta d+}
\end{align}
\end{subequations}

\item[(DB$^{-}$)] The left action on $\e_k$ with $k<0$ satisfies:
\begin{subequations}
\begin{align}
(T_i - 1)(T_i + q^{-1}) = 0, \quad T_{i}T_{i+1}T_{i} &= T_{i+1}T_{i}T_{i+1}, \quad T_{i}T_{j} = T_{j}T_{i}\, (|i - j| > 1) \label{eq:negative hecke relns} \\
T_{i}^{-1}z_{i+1}T_{i}^{-1} &= qz_{i} \, (1 \leq i \leq k-1) \label{eq:negative T and z} \\
z_{i}T_{j} = T_{j}z_{i} \, (i \not\in \{j, j+1\}),& \qquad z_{i}z_{j} = z_{j}z_{i} (1 \leq i, j \leq k) \label{eq:negative affine Hecke relns}  \\
d^{2}_{+}T_{|k|-1} = d^{2}_{+} \, (k < -1), & \qquad d_{+}T_{i} = T_{i}d_{+} \, (1 \leq i \leq |k|-2) \label{eq:negative T d+} \\
T_{1}d_{-}^{2} = d_{-}^{2},& \qquad d_{-}T_{i} = T_{i+1}d_{-} \, (1 \leq i \leq |k|-1) \label{eq:negative T d-} \\
q^{-1}\varphi d_{+} = d_{+}\varphi T_{|k|-1} , (k < -1),& \qquad T_{1}\varphi d_{-} = q^{-1}d_{-}\varphi \label{eq:negative phi} \\
z_{i}d_{+} = d_{+}z_{i},& \qquad d_{-}z_{i} = z_{i+1}d_{-} \label{eq: negative d z} \\
z_{1}(q^{-1}d_{-}d_{+} - d_{+}d_{-}) &= q^{-1}t^{-1}(d_{-}d_{+} - d_{+}d_{-})z_{|k|} \label{eq:qphi negative} \\
[\Delta_{p_m}, \Delta_{p_n}] = [\Delta_{p_m}, T_i] = &  [\Delta_{p_m}, z_j] = [\Delta_{p_m}, d_{+}] = 0 \label{eq:negative delta} \\
[\Delta^{\ast}_{p_m}, \Delta^{\ast}_{p_n}] = [\Delta^{\ast}_{p_m}, T_i] = & [\Delta^{\ast}_{p_m}, z_j] =  [\Delta^{\ast}_{p_m}, d_{+}] = 0 \label{eq:negative delta star} \\
[\Delta^{\ast}_{p_m}, \Delta_{p_n}] & = 0 \label{eq: negative delta delta star} \\
[\Delta_{p_m}, d_{-}] = -z_1^{m}d_{-}, & \qquad [\Delta_{p_m}^{\ast}, d_{-}] = -z_1^{-m}d_{-}. \label{eq: negative delta d-}
\end{align}
\end{subequations}
\end{enumerate}

\begin{remark}
Note that none of the relations involve the element $\varphi$ acting on $\e_0$, i.e. none of them involve $\e_0[d_{+},d_{-}]\e_0$. 
\end{remark}

\begin{remark}
    For $k > 0$, relations \eqref{eq:hecke relns}, \eqref{eq:T and z} and \eqref{eq:affine Hecke relns} say that $\e_k z_i\e_k$ and $\e_k T_j\e_k$ generate a copy of the affine Hecke algebra $\AH_k$ with parameter $q$. Similarly, for $k < 0$ the relations \eqref{eq:negative hecke relns}, \eqref{eq:negative T and z} and \eqref{eq:negative affine Hecke relns} say that $\e_k z_i\e_k$ and $\e_k T_j\e_k$ generate a copy of the affine Hecke algebra $\AH_{|k|}$ with parameter $q^{-1}$. 
\end{remark}

To express the last relation, we need to define some additional elements in $\DBqt$. 

\begin{definition}\label{def:e-and-f}
For each $m \in \Z$, we define the elements 
\begin{align}\label{eq:def-e-and-f}
e_m := \e_0 d_{-} z_1^{m} d_{+}\e_0, \qquad f_m := \e_0 d_{+}z_1^{m}d_{-} \e_0,
\end{align}
with corresponding generating series given by
\[
e(\z) = \sum_{i \in \Z} e_i\z^{-i}, \qquad f(\w) = \sum_{i \in \Z}f_i\w^{-i}.
\]

We also define elements $\psi^{\pm m},m\ge 0$ via the generating series $\psi^{\pm}(\z) = \sum_{m \geq 0}\psi^{\pm}_{m}\z^{\mp m}$, where: 
\begin{align}
\label{eq: def psi plus}
\psi^+(\z)&:=\left(-\frac{1-q^{-1}t^{-1}\z^{-1}}{1-\z^{-1}}\right)\exp\left[-\sum_{m=1}^{\infty}\frac{\z^{-m}}{m}\e_0\Delta_{p_m}\e_0(1-q^m)(1-t^m)(1-(qt)^{-m})\right] \\
\label{eq: def psi minus}
\psi^-(\z)&:=\left(-\frac{1-q^{-1}t^{-1}\z^{-1}}{1-\z^{-1}}\right)\exp\left[\sum_{m=1}^{\infty}\frac{\z^m}{m}\e_0\Delta^*_{p_m}\e_0(1-q^m)(1-t^m)(1-(qt)^{-m})\right]
\end{align} 
\end{definition}

\begin{remark}
    The $\left(-\frac{1-q^{-1}t^{-1}\z^{-1}}{1-\z^{-1}}\right)$ factor in \eqref{eq: def psi plus} and \eqref{eq: def psi minus} is motivated by Corollary \ref{cor: feigin tsymbaliuk psi exponential}, see also Theorem \ref{thm:dbqt-polynomial-rep}.
\end{remark}
Then, the last relation in $\DBqt$ states:
\begin{enumerate}
\item[(DB$^{0}$)] The elements $e_m, f_m$ and $\psi^\pm_m$ satisfy,
\begin{equation}
\label{eq: weak e and f commutation in double}
[e(\z),f(\w)]= (1-q^{-1}t^{-1})^{-1} \left(\sum_{n\in \Z}\z^n\w^{-n}\right) (\psi^{+}(\w) - \psi^{-}(\z))
\end{equation}
\end{enumerate}
In particular, notice that \eqref{eq: weak e and f commutation in double} is very similar to \eqref{eq: commutation of e and f}, except for a $-(1-q)^{-1}(1-t)^{-1}$ factor. Thus, \eqref{eq: weak e and f commutation in double} can be equivalently written as an infinite collection of relations as in \eqref{eq:explicit-commutation-e-and-f}. The missing $(1-q)^{-1}(1-t)^{-1}$ factor is explained by Corollary \ref{cor: e and f short}, see also Theorem \ref{thm:dbqt-polynomial-rep}.  
\end{definition}
\medskip

The \emph{double Dyck path algebra} $\Bqt$ was originally introduced by Carlsson-Gorsky-Mellit in \cite{CGM}. In \cite{calibrated} we extended 
this definition to include a family of commuting operators $\{\Delta_{p_m}\}$ that act like Macdonald operators on the polynomial representation. These algebras can be realized inside $\DBqt$ as follows. 
\begin{definition}
The \newword{double Dyck path algebra} $\Bqt$ is the subalgebra of $\DBqt$ generated by the arrows and vertices in the following subquiver of $Q_{\Z}$:
\[\begin{tikzcd}
	0 && 1 && 2 && \cdots
	\arrow["{d_{+}}", curve={height=-6pt}, shorten <=5pt, shorten >=5pt, from=1-1, to=1-3]
	\arrow["{d_-}", curve={height=-6pt}, shorten <=5pt, shorten >=5pt, from=1-3, to=1-1]
	\arrow["z"{description}, from=1-3, to=1-3, loop, in=240, out=300, distance=5mm]
	\arrow["{d_{+}}", curve={height=-6pt}, shorten <=5pt, shorten >=5pt, from=1-3, to=1-5]
	\arrow["{d_-}", curve={height=-6pt}, shorten <=5pt, shorten >=5pt, from=1-5, to=1-3]
	\arrow["z"{description}, from=1-5, to=1-5, loop, in=240, out=300, distance=5mm]
	\arrow["T"{description}, from=1-5, to=1-5, loop, in=230, out=310, distance=15mm]
	\arrow["{d_{+}}", curve={height=-6pt}, shorten <=5pt, shorten >=5pt, from=1-5, to=1-7]
	\arrow["{d_-}", curve={height=-6pt}, shorten <=5pt, shorten >=5pt, from=1-7, to=1-5]
\end{tikzcd}\]
\end{definition}

 Its extension by Delta operators can thus be seen as follows.

\begin{definition}
Define the \newword{extended double Dyck path algebra} $\DBqt^{\geq}$ is the subalgebra of $\DBqt$ generated by the arrows and vertices of the following quiver. 

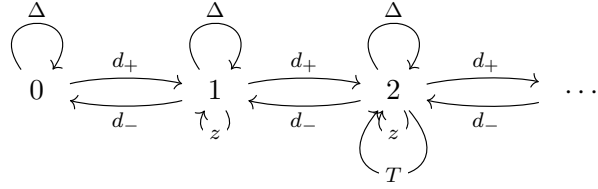
\[\begin{tikzcd}
	0 && 1 && 2 && \cdots
	\arrow["\Delta", from=1-1, to=1-1, loop, in=55, out=125, distance=10mm]
	\arrow["{d_{+}}", curve={height=-6pt}, shorten <=5pt, shorten >=5pt, from=1-1, to=1-3]
	\arrow["{d_-}", curve={height=-6pt}, shorten <=5pt, shorten >=5pt, from=1-3, to=1-1]
	\arrow["\Delta", from=1-3, to=1-3, loop, in=55, out=125, distance=10mm]
	\arrow["z"{description}, from=1-3, to=1-3, loop, in=240, out=300, distance=5mm]
	\arrow["{d_{+}}", curve={height=-6pt}, shorten <=5pt, shorten >=5pt, from=1-3, to=1-5]
	\arrow["{d_-}", curve={height=-6pt}, shorten <=5pt, shorten >=5pt, from=1-5, to=1-3]
	\arrow["\Delta", from=1-5, to=1-5, loop, in=55, out=125, distance=10mm]
	\arrow["z"{description}, from=1-5, to=1-5, loop, in=240, out=300, distance=5mm]
	\arrow["T"{description}, from=1-5, to=1-5, loop, in=230, out=310, distance=15mm]
	\arrow["{d_{+}}", curve={height=-6pt}, shorten <=5pt, shorten >=5pt, from=1-5, to=1-7]
	\arrow["{d_-}", curve={height=-6pt}, shorten <=5pt, shorten >=5pt, from=1-7, to=1-5]
\end{tikzcd}\]

In particular, $\DBqt^{\geq}$ is an enlargement of the $\Bqt^{\ext}$-algebra first defined in \cite{calibrated}. In $\Bqt^{\ext}$  the operators $\Delta^*_{p_m}$ were not included.
\end{definition}

\begin{definition}
The algebra $\DBqt^{\leq}$ is the subalgebra of $\DBqt$ generated by the following quiver:

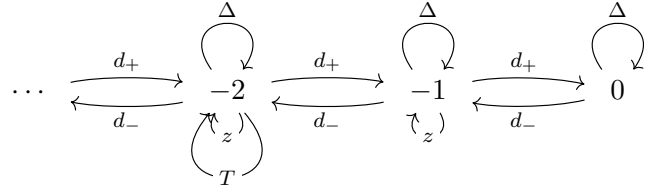
\[\begin{tikzcd}
	\cdots && {-2} && {-1} && 0
	\arrow["{d_{+}}", curve={height=-6pt}, shorten <=5pt, shorten >=5pt, from=1-1, to=1-3]
	\arrow["{d_-}", curve={height=-6pt}, shorten <=5pt, shorten >=5pt, from=1-3, to=1-1]
	\arrow["\Delta", from=1-3, to=1-3, loop, in=55, out=125, distance=10mm]
	\arrow["z"{description}, from=1-3, to=1-3, loop, in=240, out=300, distance=5mm]
	\arrow["T"{description}, from=1-3, to=1-3, loop, in=230, out=310, distance=15mm]
	\arrow["{d_{+}}", curve={height=-6pt}, shorten <=5pt, shorten >=5pt, from=1-3, to=1-5]
	\arrow["{d_-}", curve={height=-6pt}, shorten <=5pt, shorten >=5pt, from=1-5, to=1-3]
	\arrow["\Delta", from=1-5, to=1-5, loop, in=55, out=125, distance=10mm]
	\arrow["z"{description}, from=1-5, to=1-5, loop, in=240, out=300, distance=5mm]
	\arrow["{d_{+}}", curve={height=-6pt}, shorten <=5pt, shorten >=5pt, from=1-5, to=1-7]
	\arrow["{d_-}", curve={height=-6pt}, shorten <=5pt, shorten >=5pt, from=1-7, to=1-5]
	\arrow["\Delta", from=1-7, to=1-7, loop, in=55, out=125, distance=10mm]
\end{tikzcd}\]    
\end{definition}

\begin{remark}
    Comparing the relations (DB$^+$) with (DB$^-$), it follows from the definitions that $\DBqt^{\leq}$ is isomorphic to $\mathbb{DB}_{q^{-1}, t^{-1}}^{\geq}$. 
\end{remark}

\subsection{An Anti-Involution on $\DBqt$} 
Given the obvious symmetry between $\DBqt^{\geq}$ and $\DBqt^{\leq}$ it is natural to search for a map on $\DBqt$ which exchanges these subalgebras. In this section, we define such an anti-involution $\adj: \DBqt \to \DBqt$. This map is similar to the anti-linear anti-involution $\adj: \Bqt^{\ext} \to \Bqt^{\ext}$ previously defined in \cite[Section 7]{calibrated} with some subtle differences, see Remark \ref{rmk:diff-previous-paper}.

\begin{lemma}\label{lem:antiinvolution-double}
The following defines a $\C(q,t)$-linear anti-involution $\adj: \DBqt \to \DBqt$:
\begin{align*}
    \adj(\e_k) = \e_{-k} & &  \adj(\e_kT_i\e_k) = \e_{-k}T^{-1}_{|k|-i}\e_{-k} & & \adj(\e_kz_j\e_k) = \e_{-k}z_{|k|+1-j}\e_{-k}
\end{align*}
\begin{align*}
\adj(\e_{k+1}d_{+}\e_{k}) = \begin{cases} q^{k/2}\e_{-k}d_{+}\e_{-k-1} & k \geq 0 \\ q^{(k+1)/2}\e_{-k}d_{+}\e_{-k-1} & k < 0 \end{cases}  \\
\adj(\e_{k-1}d_{-}\e_{k}) = \begin{cases}  q^{(k-1)/2}\e_{-k}d_{-}\e_{-k+1} & k > 0  \\ q^{k/2}\e_{-k}d_{-}\e_{-k+1} & k \leq 0 \end{cases}
\end{align*}
\begin{align*}
\adj(\e_k\Delta_{p_m}\e_k) = \e_{-k}(\Delta_{p_m} + \sum_{i = 1}^{|k|}z_i^{m})\e_{-k} & & \adj(\e_k\Delta^{\ast}_{p_m}\e_k) = \e_{-k}(\Delta^{\ast}_{p_m} + \sum_{i = 1}^{|k|}z_i^{-m})\e_{-k}.
\end{align*}
In particular, 
\begin{equation}
\label{eq: Theta for e and f}
\Theta(e_k)=f_k,\  \Theta(\psi_m^{\pm})=\psi_m^{\pm}.
\end{equation}
Moreover,
$\Theta(\DBqt^{\geq}) = \DBqt^{\leq}$ and vice versa.  
\end{lemma}

\begin{remark}\label{rmk:diff-previous-paper}
    The map $\adj$ is closely related to, but exhibits some subtle differences with, the map from \cite[Lemma 7.2]{calibrated}, that we call $\widetilde{\adj}$. To start, $\adj$ is $\C(q,t)$-linear while $\widetilde{\adj}$ is $\theta$-linear, where $\theta: \C(q,t) \to \C(q,t)$ is defined by $\theta(q) = q^{-1}, \theta(t) = t^{-1}$. The reason is that the relations (DB$^-$) defining $\DBqt$ already take $\theta$ into account. The second difference is that $\widetilde{\adj}$ interchanges the generators $d_{+}$ and $d_{-}$, while $\adj$ preserves them. Taking a close look at the relations (DB$^-$), however, one sees that the relations defining $\DBqt$ also take this into account. Finally, $\widetilde{\adj}$ sends $\Delta_{p_m}$ to $-\Delta_{p_m} + z_1^{m} + \cdots + z_k^{m}$, which differs from $\adj$ by the sign of $\Delta_{p_m}$. This difference is explained by looking at the relations \eqref{eq: delta d+} and \eqref{eq: negative delta d-}, which differ by a sign. 
\end{remark}

\begin{proof}[Proof of Lemma \ref{lem:antiinvolution-double}]
First, note that the map $z_i \mapsto z_{k+1-i}$, $T_i \mapsto T_{k-i}^{-1}$ gives an anti-homomorphism of affine Hecke algebras $\AH_k(q) \to \AH_k(q^{-1})$ that squares to the identity, so that $\adj$ interchanges the relations \eqref{eq:hecke relns}, \eqref{eq:T and z} and \eqref{eq:affine Hecke relns} with the relations \eqref{eq:negative hecke relns}, \eqref{eq:negative T and z} and \eqref{eq:negative affine Hecke relns}, respectively.

Similarly, direct computation shows that relations \eqref{eq:T d-}, \eqref{eq:T d+}, and \eqref{eq: d z} are exchanged under $\adj$ with relations \eqref{eq:negative T d-}, \eqref{eq:negative T d+}, and \eqref{eq: negative d z} respectively. 

Before moving on, we find some values of $\adj$.

\emph{If $k > 0$:}
\begin{align*}
\adj(\e_k(d_{+}d_{-} - d_{-}d_{+})\e_k) &= \e_{-k}(q^{k-1}d_{-}d_{+} - q^{k}d_{+}d_{-})\e_{-k} = q^{k}\e_{-k}(q^{-1}d_{-}d_{+} - d_{+}d_{-})\e_{-k} 
\\
\adj(\e_k(qd_{+}d_{-} - d_{-}d_{+}) &= q^{k}\e_{-k}(d_{-}d_{+} - d_{+}d_{-})\e_{-k}
\end{align*}

\emph{If $k < 0$}:
\begin{align*}
\Phi(\e_k(d_{-}d_{+} - d_{+}d_{-})\e_k) &= q^{k}\e_{-k}(qd_{+}d_{-} - d_{-}d_{+})\e_{-k}
\\
\Phi(\e_k(q^{-1}d_{-}d_{+} - d_{+}d_{-})\e_k) &= q^{k}\e_{-k}(d_{+}d_{-} - d_{-}d_{+})\e_{-k}.
\end{align*}

Thus, applying $\adj$ to both sides of \eqref{eq:qphi} we obtain
\[
\adj(\e_kz_{1}(qd_{+}d_{-} - d_{-}d_{+})\e_k) = q^{k}\e_{-k}(d_{-}d_{+} - d_{+}d_{-})z_{k}\e_{-k}
\]
and
\[
\adj(qt\e_k(d_{+}d_{-} - d_{-}d_{+})z_{k}\e_k) = qt\e_{-k}q^{k}z_1(q^{-1}d_{-}d_{+} - d_{-}d_{+})\e_{-k}.
\]
Hence, we see that $\adj$ sends relation \eqref{eq:qphi} to relation \eqref{eq:qphi negative}. An analogous computation shows that \eqref{eq:qphi negative} is exchanged with \eqref{eq:qphi}.

Now we ned to verify that $\adj$ interchanges relations \eqref{eq:phi} and \eqref{eq:negative phi}. Applying $\adj$ to the right-hand side of the left equation in \eqref{eq:phi} we obtain
\begin{align*}
\adj(\e_{k-1}q[d_{+},d_{-}]d_{-}\e_k) = & qq^{(k-1)/2}q^{k-1}\e_{-k}d_{-}(q^{-1}d_{-}d_{+} - d_{+}d_{-})\e_{-k+1}\\
= & q^{\frac{3k-1}{2}}\e_{-k}d_{-}(q^{-1}d_{-}d_{+} - d_{+}d_{-})\e_{-k+1} \\ 
= & q^{\frac{3k-1}{2}}\e_{-k}d_{-}(q^{-1}d_{-}d_{+} - q^{-1}d_{+}d_{-} + q^{-1}d_{+}d_{-} - d_{+}d_{-})\e_{-k+1}\\
= & q^{\frac{3k-1}{2}}\e_{-k}(q^{-1}d_{-}[d_{-},d_{+}] + (q^{-1}-1)d_{-}d_{+}d_{-})\e_{-k+1}
\end{align*}
On the other hand, applying $\adj$ to the left-hand side of the same equation we obtain
\begin{align*}
\adj(\e_{k-1}d_{-}[d_{+},d_{-}]T_{k-1}\e_{k}) = & q^{k}q^{(k-1)/2}\e_{-k}T_1^{-1}(q^{-1}d_{-}d_{+} - d_{+}d_{-})d_{-}\e_{-k+1} \\
= & q^{\frac{3k-1}{2}}\e_{-k}(q^{-1}T_1^{-1}[d_{-}, d_{+}]d_{-} + (q^{-1}-1)T_1^{-1}d_{+}d_{-}d_{-})\e_{-k+1} \\
= & q^{\frac{3k-1}{2}}\e_{-k}(T_1[d_{-}, d_{+}]d_{-} + (q^{-1}-1)([d_{-}, d_{+}]d_{-} + d_{+}d_{-}d{-}))\e_{-k+1} \\
= & q^{\frac{3k-1}{2}}\e_{-k}(T_1[d_{-}, d_{+}]d_{-} + (q^{-1}-1)d_{-}d_{+}d_{-})\e_{-k+1}
\end{align*}
where we used that $q^{-1}T_1^{-1} = T_1 + q^{-1} - 1$ (recall that, since $-k < 0$, the correct quadratic relation for $T_1$ is $(T_1 - 1)(T_1 + q^{-1}) = 0$), that $T_1^{-1}d_{+}= d_{+}T_{1}^{-1}$ and that $T_1^{-1}d_{-}^{-2} = d_{-}^{2}$. Thus, we see that $\adj$ sends the left equation in \eqref{eq:phi} to the right equation in \eqref{eq:negative phi}. Analogously, we see that $\adj$ sends the right equation in \eqref{eq:negative phi}, and that it exchanges the right side of \eqref{eq:phi} with the left side of \eqref{eq:negative phi}.

Now we verify relations \eqref{eq:delta}, \eqref{eq:delta star}, \eqref{eq: delta delta star} and \eqref{eq: delta d+}. Since $z_1^{m} + \cdots + z_k^{m}$ is central in the affine Hecke algebra $\AH_k(q^{\pm 1})$, it is clear that the first three of these relations are preserved under $\adj$, with the exception of $[\Delta_{p_m}, d_{-}] = 0$ that we verify next. We have
\begin{align*}
\adj(\e_{k-1}(\Delta_{p_m}d_{-})\e_{k}) = & q^{(k-1)/2}\e_{-k}(d_{-}(\Delta_{p_m} + z_1^{m} + \cdots + z_{k-1}^{m})\e_{-k+1} \\
= & q^{(k-1)/2}\e_{-k}(d_{-}\Delta_{p_m} + d_{-}(z_1^{m} + \cdots + z_{k-1}^{m}))\e_{-k+1} \\
= & q^{(k-1)/2}\e_{-k}(\Delta_{p_m}d_{-} + z_1^{m}d_{-} + (z_2^{m} + \cdots + z_k^{m})d_{-})\e_{-k+1} \\
= & \adj(\e_{k-1}(d_{-}\Delta_{p_m})\e_{k}).
\end{align*}

Let us now check that the relation \eqref{eq: delta d+} is preserved. We can write this relation as
\[
\e_{k+1}\Delta_{p_m}d_{+}\e_{k} = \e_{k+1}d_{+}\Delta_{p_m}\e_{k} + \e_{k+1}z_1^md_{+}\e_k. 
\]
Applying $\adj$ to the left-hand side:
\begin{align*}
\adj(\e_{k+1}\Delta_{p_m}d_{+}\e_k)  & =  q^{\frac{k}{2}}\e_{-k}d_{+}(\Delta_{p_m} + z_1^{m} + \cdots + z_{k+1}^{m})\e_{-k-1} \\
    & = q^{\frac{k}{2}}\e_{-k}(\Delta_{p_m}d_{+} + d_{+}(z_1^{m} + \cdots + z_{k+1}^{m}))\e_{-k-1} \\
    & = q^{\frac{k}{2}}\e_{-k}(\Delta_{p_m} + z_1^{m} + \cdots + z_{k}^{m})d_{+}\e_{-k-1} + q^{\frac{k}{2}}\e_{-k}d_{+}z^{m}_{k+1}\e_{-k-1} \\
    & = \adj(\e_{k+1}d_{+}\Delta_{p_m}\e_k) + \adj(\e_{k+1}z_1^{m}d_{+}\e_k),
\end{align*}
and the result follows. Similarly, one can check that $\adj$ sends \eqref{eq: negative delta d-} to \eqref{eq: delta d+}.

Finally, we need to check that (DB$^0$) is preserved. To this end, note that $\adj(e_m) = f_m$, $\adj(f_k) = e_k$ and $\adj(\e_0\Delta_{p_m}\e_0) = \e_0\Delta_{p_m}\e_0$, $\adj(\e_0\Delta^{\ast}_{p_m}\e_0) = \e_0\Delta^{\ast}_{p_m}\e_0$, so that $\adj$ fixes $\psi_m$ for every $m$ and the preservation of (DB$^0$) follows.
\end{proof}

\subsection{Polynomial Representation of $\DBqt^{\geq}$}\label{sec:polynomial-bqt} The polynomial representation of $\Bqt$ appeared originally in the work of Carlsson and Mellit in \cite{Shuffle} and thereafter, alongside Gorsky, geometrically formulated in \cite{CGM}. Following \cite{calibrated}, we recall the extension of the polynomial representation $V=\bigoplus_{k=0}^{\infty} V_k$ to $\DBqt^{\geq}$. For convenience, we relabel this representation as $V^{+} = \bigoplus_{k = 0}^{\infty}V_k^{+}$. The main theorem in \cite{CGM} gives a convenient basis of $V^{+}$ and the action of the generators in this basis.

Consider a pair of Young diagrams $\lambda\subseteq \mu$ such that $\mu\setminus \lambda$ is a union of horizontal strips, and a standard Young tableau $T$ of skew shape $\mu\setminus \lambda$. Recall, that for a box $\Box \in \lambda$ in row $r$ and column $c$ the \emph{(q,t)-content} of this box is given by $q^{c-1}t^{r-1}$. Let $w_i$ be the $(q,t)$-content for a box labeled by $|\mu \setminus \lambda| - i +1$ in $T$, and denote the skew tableau $T$ by the pair $(\lambda, \underline{w})$, where $\underline{w}
 = (w_{|\mu\setminus\lambda|}, \dots, w_1)$. The basis $\{I_{\lambda, \underline{w}}\}$ of $V_k^{+}$ is given by all such tableaux $(\lambda,\underline{w})$ of size $k$.  
 When $k = 0$, the basis $\{I_{\lambda}\}$ of $V_{0}^{+} = \Sym_{q,t}$ coincides with the modified Macdonald basis in $\Sym_{q,t}$, see \cite[Theorem 7.0.1]{CGM}. More generally, it was recently shown that the basis $\{I_{\lambda, \underline{w}}\}$ of $V_k^{+} = \Sym_{q,t}[y_1, \dots, y_k]$ coincides with the basis of modified partially symmetric Macdonald polynomials, see \cite{OrrMilo}. 

On the basis $\{I_{\lambda, \underline{w}}\}$ the action of $\DBqt^{\geq}$  is given by:
\begin{subequations}
\begin{align}
\label{eq: z polynomial}
z_{j}I_{\lambda, w} &= w_{j}I_{\lambda, w}, 
\\
\label{eq: Delta polynomial}
\Delta_{p_m}I_{\lambda, w} &= \left(\sum_{\square\in \lambda}\square^m+\sum_{j=1}^{k}w_{j}^m\right)I_{\lambda, w},\\
\label{eq: Delta* polynomial}
\Delta_{p_m}^*I_{\lambda, w} &= \left(\sum_{\square\in \lambda}\square^{-m}+\sum_{j=1}^{k}w_{j}^{-m}\right)I_{\lambda, w},
\\
\label{eq: T polynomial}
 T_{i}I_{\lambda, w} &= \frac{(q-1)w_{i+1}}{w_{i} - w_{i+1}}I_{\lambda, w} + \frac{w_{i} - qw_{i+1}}{w_{i} - w_{i+1}}I_{\lambda, s_{i}w},
\\
\label{eq: d- polynomial}
d_{-}I_{\lambda, (w_k, \dots, w_{1})} &= I_{\lambda\cup w_k, (w_{k-1}, \dots, w_{1})},
\\
\label{eq: d+ polynomial}
d_{+}I_{\lambda, (w_k, \dots, w_{1})} &= -q^{k}\sum_{\substack{x \; \text{an addable} \\ \text{box of} \; \mu}} c(\lambda\cup \underline{w};x)\prod_{i = 1}^{k}\frac{x - tw_i}{x-qtw_i}I_{\lambda, (w_k, \dots, w_1,x)}, 
\end{align}
\end{subequations}
where we abuse the notation by identifying an addable box $x$ with its $(q,t)$-content, and  the coefficients $c(\lambda;x)$ are given by \eqref{eq: c for partitions}. Recall the element $e_m$ from \eqref{eq:def-e-and-f}. Note that
\[
e_mI_{\lambda} = d_{-}z_1^{m}d_{+}I_{\lambda} = d_{-}z_1^{m}\sum_{x}c(\lambda;x)I_{\lambda; x} = \sum_{\mu = \lambda \cup x}c(\lambda;x)x^{m}I_{\mu}.
\]
Comparing to Corollary \ref{cor: e and f short}, we obtain the following result.

\begin{lemma}\label{lem:comparison-e-action}
The action of $e_m$ on $V_0^+$ agrees, up to the scalar $-(1-q)(1-t)$, with the action of its eponym in the elliptic Hall algebra $\EHA$. More precisely, if $e_m$ denotes the action of $e_m \in \DBqt^{\geq}$ and $\tilde{e}_m$ the action of its eponymic generator in $\EHA$, then
\[
e_m = -(1-q)(1-t)\tilde{e}_m. 
\]
\end{lemma}

\subsection{Polynomial Representation of $\DBqt^{\leq}$}\label{sec:polynomial-dbqt} Similarly to the polynomial representation of $\DBqt^{\geq}$, we define a polynomial representation of the non-positive half $\DBqt^{\leq}$ as follows. We set:
\[
V^{-} := \bigoplus_{k =0}^{\infty}V^{-}_{-k}
\]
Where $V^{-}_{-k}$ is defined as follows. Consider a pair of Young diagrams $\mu \subseteq \lambda$ such that $\lambda \setminus \mu$ is a horizontal strip and $|\lambda \setminus \mu| = -k$. Consider also a standard tableau $T$ of shape $\lambda \setminus \mu$. We denote by $w_i$ the $(q,t)$-content of the box labeled by $|\lambda \setminus \mu| - i + 1$ in $T$, and we denote the tableau $T$ by the pair $(\lambda, \underline{w})$. Note that, as opposed to the setting in Section \ref{sec:polynomial-bqt}, here the box $w_{k}$ is a \emph{removable} box of $\lambda$; $w_{k-1}$ is a removable box of $\lambda - w_{k}$ and so on. The space $V_k$ has a $\C(q,t)$-basis given by $I^{-}_{\lambda, \underline{w}}$, for all such tableaux $T$.

In this basis, we have:
\begin{subequations}
\begin{align}
\label{eq: z polynomial negative}
z_{j}I^{-}_{\lambda, w} &= w_{j}I^{-}_{\lambda, w}, 
\\
\label{eq: Delta polynomial negative}
\Delta_{p_m}I^{-}_{\lambda, w} &= \left(\sum_{\square\in \lambda}\square^m-\sum_{j=1}^{k}w_{j}^m\right)I^{-}_{\lambda, w},\\
\label{eq: Delta* polynomial negative}
\Delta_{p_m}^*I_{\lambda, w} &= \left(\sum_{\square\in \lambda}\square^{-m}-\sum_{j=1}^{k}w_{j}^{-m}\right)I^{-}_{\lambda, w},
\\
\label{eq: T polynomial negative}
 T_{i}I^{-}_{\lambda, w} &= \frac{(q^{-1}-1)w_{i+1}}{w_{i} - w_{i+1}}I^{-}_{\lambda, w} + \frac{w_{i} - q^{-1}w_{i+1}}{w_{i} - w_{i+1}}I^{-}_{\lambda, s_{i}w},
\\
\label{eq: d+ polynomial negative}
d_{+}I^{-}_{\lambda, (w_k, \dots, w_{1})} &= I^{-}_{\lambda- w_k, (w_{k-1}, \dots, w_{1})},
\\
\label{eq: d- polynomial negative}
d_{-}I^{-}_{\lambda, (w_k, \dots, w_{1})} &= -q^{k}\sum_{\substack{x \; \text{a removable} \\ \text{box of} \; \mu}} c^{\ast}(\lambda- \underline{w};x)\prod_{i = 1}^{k}\frac{x - t^{-1}w_i}{x-q^{-1}t^{-1}w_i}I^{-}_{\lambda, (w_k, \dots, w_1,x)}, 
\end{align}
\end{subequations}

\noindent where $c^{\ast}(\lambda - \underline{w}; x)$ is as defined in \eqref{eq: c* for partitions}. 

\begin{lemma}
The equations \eqref{eq: z polynomial negative}--\eqref{eq: d- polynomial negative} define a representation of $\DBqt^{\leq}$ on $V^{-}$. 
\end{lemma}
\begin{proof}
Since $\DBqt^{\leq} \cong \mathbb{DB}^{\geq}_{q^{-1}, t^{-1}}$ this follows from \cite[Theorem 3.25]{calibrated}, as follows. The poset $E$ required in \emph{loc. cit.} is the opposite Young poset, and we need to verify that the coefficients $c^{\ast}(\lambda; x)$ satisfy the  monodromy conditions for $q^{-1}, t^{-1}$. Since
\[
-\frac{(x-t^{-1}y)(x-q^{-1}y)(y - q^{-1}t^{-1}x)}{(y - t^{-1}x)(y - q^{-1}x)(x - t^{-1}q^{-1}y)} = -\frac{(y - tx)(y - qx)(x - qty)}{(x - ty)(x - qy)(y - qtx)}
\]
the result now follows from Lemma \ref{lem:dual-monodromy-conditions}. 
\end{proof}

\begin{remark}
The representation $V^{-}$ has the following natural interpretation in terms of the map $\adj$. For any $\DBqt^{\geq}$-representation $W$, the $\C(q,t)$-dual $W^{\ast}$ becomes a $\DBqt^{\leq}$-representation via $\adj$. In the case of the polynomial representation $V = \bigoplus_{k = 0}^{\infty}V_k$ we can do better. Instead of looking at the full dual $V^{\ast}$, we look at the space of locally finite (for the action of $z$'s and $\Delta$'s) vectors in the graded dual $\bigoplus_{k = 0}^{\infty} V_k^{\ast}$, which defines a $\DBqt^{\leq}$-subrepresentation
\[
V^{\ast}_{l.f.} = \bigoplus_{k = 0}^{\infty} V_{k, l.f.}^{\ast}\subseteq \bigoplus_{k = 0}^{\infty} V_k^{\ast}. 
\]
The polynomial representation $V$ of $\DBqt^{\geq}$ is calibrated \cite{calibrated} and, moreover, it is the representation associated to the Young poset with weighting given by the sum of the $(q,t)$-contents of boxes, see \cite[Section 3.6.3]{calibrated}. By \cite[Theorem 7.10]{calibrated}, the space $V^{\ast}_{l.f.}$ is the calibrated representation associated to the opposite Young poset. The basis $I^{-}_{\lambda, (w_k, \dots, w_1)}$ corresponds to a rescaling of the dual basis $I_{\mu, (w_1, \dots, w_k)}^{\ast}$, see \cite[Theorem 3.26 and Lemma 3.29]{calibrated}. 
\end{remark}

Similarly to Lemma \ref{lem:comparison-e-action} we obtain the following result.

\begin{lemma}\label{lem:comparison-f-action}
The action of $f_m \in \DBqt^{\leq}$ defined by \eqref{eq:def-e-and-f} on $V_0$ agrees with the action of its eponym in $\EHA$.
\end{lemma}

\subsection{Polynomial Representation of $\DBqt$} We now define the polynomial representation of $\DBqt$, by gluing the $\DBqt^{\geq}$-representation $V^{+}$ and the $\DBqt^{\leq}$-representation $V^{-}$ along their zeroth component. We define:
\[
V = (V^{+} \oplus V^{-})/\langle I_{\lambda} - I^{-}_{\lambda}; \lambda \in P\rangle. 
\]
Note that $V = \bigoplus_{k = -\infty}^{\infty}V_k$, where
\[
V_k = \begin{cases}V_k^{+}, & k > 0 \\ V_k^{-}, & k < 0 \\ (V_0^{+} \oplus V_0^{-})/\langle I_{\lambda} - I^{-}_{\lambda}; \lambda \in P\rangle, & k = 0. \end{cases}
\]

\begin{theorem}\label{thm:dbqt-polynomial-rep}
The formulas \eqref{eq: z polynomial}--\eqref{eq: d+ polynomial} and \eqref{eq: z polynomial negative}--\eqref{eq: d- polynomial negative} define a representation of $\DBqt$ on $V$.
\end{theorem}
\begin{proof}
    We only need to check the relation \eqref{eq: weak e and f commutation in double}. From Lemma \ref{lem:comparison-e-action}, we see that the operators $e_m$ on $V_0$ agree, up to the scalar $-(1-q)(1-t)$, with the action of the eponymic elliptic Hall algebra operators of Feigin-Tsymbaliuk on $V_0$, with a similar result for $f_m$ by Lemma \ref{lem:comparison-f-action}. By Corollary \ref{cor: feigin tsymbaliuk psi exponential} and their definition (\eqref{eq: def psi plus}, \eqref{eq: def psi minus}), the action of the $\psi_m^{\pm}$-operators also agrees with their action of their namesake elements of the elliptic Hall algebra. At this point the result follows since, by \cite[Theorem 3.1]{FT}, the operators coming from the elliptic Hall algebra satisfy the relation \eqref{eq: commutation of e and f}. Hence, the operators coming from the double Dyck path algebra also satisfy \eqref{eq: weak e and f commutation in double}.   
\end{proof}

The proof of \ref{thm:dbqt-polynomial-rep} hints at a very close connection between the elliptic Hall algebra $\EHA$ and the algebra $\DBqt$. Indeed, we have the following result.

\begin{lemma}\label{lem: polynomial reps agree}
Consider the polynomial representation $V$ of $\DBqt$ and the polynomial representation $V_0 \subseteq V$ of $\EHA$. The action of the elements $e_m, f_m, \psi_m^{\pm}$ of $\e_0\DBqt\e_0$ on $V_0$ agree, up to a scalar, with the action of their eponyms in $\EHA$.
\end{lemma}
 
\section{Verification of the EHA relations in the Dyck path algebra}
\label{sec: EHA relations}

In this section we show that the elements $e_m, \psi^{\pm}_{m}$ and $f_m$ satisfy the relations of the elliptic Hall algebra, and thus $\EHA$ arises as a subalgebra of $\DBqt$.
\smallskip

Recall from Definition \ref{def:e-and-f} the elements $e_m = \e_0d_{-}z_1^md_{+}\e_0$ and $f_m = \e_0d_{+}z_1^md_{-}\e_0 \in \e_0\DBqt\e_0$. 
We also recall the elements $\psi^{\pm}_k$, see \eqref{eq: def psi plus} and \eqref{eq: def psi minus}.

\begin{definition}
We define $\Eqt$ as the subalgebra of $\e_0\DBqt\e_0$ generated by $e_m,f_m,\Delta_{p_m}$ and $\Delta_{p_m}^*$. This algebra has several useful subalgebras:
\begin{itemize}
\item $\Eqt^{>}\subset \e_0\Bqt\e_0$ is the subalgebra generated by $e_m$. 
\item $\Eqt^{\geq}\subset \e_0\DBqt\e_0$ 
is the subalgebra generated by $e_m$ and $\Delta_{p_m}, \Delta_{p_m}^*$. 
\item $\Eqt^{<}\subset \e_0\Bqt^{-}\e_0$ is the subalgebra generated by $f_m$. 
\item $\Eqt^{\leq}\subset \e_0\DBqt\e_0$ 
is the subalgebra generated by $f_m$ and $\Delta_{p_m}, \Delta_{p_m}^*$. 
\end{itemize}
\end{definition}

Equation \eqref{eq: Theta for e and f} immediately implies:

\begin{lemma}
\label{lem: theta e f}
We have $\Theta(e_m)=f_m$ and $\Theta(f_m)=e_m$, so $\Theta$ preserves $\Eqt$ and exchanges $\Eqt^{>}$ with $\Eqt^{<}$ (resp.  $\Eqt^{\geq}$ with $\Eqt^{\leq}$).
\end{lemma}

\subsection{First Main Result}\label{sec:constructing-representations}
The following theorem is the primary result of this section:

\begin{theorem}
\label{thm: E plus}
The
elements $e_m,\psi_m^{\pm} \in \e_0\DBqt^{\geq}\e_0$ satisfy relations \eqref{eq: quadratic e}, \eqref{eq: commutation e with psi}, \eqref{eq: cubic e}, \eqref{eq: psi commute}, and \eqref{eq: psi 0 invertble}.
\end{theorem}
Equations \eqref{eq: psi commute} and \eqref{eq: psi 0 invertble} follow from definitions. Equation \eqref{eq: commutation e with psi} is proved in Lemma \ref{lem: e psi commutation abstract}, Equation \eqref{eq: quadratic e} is proved in Lemma \ref{lem:quadratic-holds-in-bqt}  and Equation \ref{eq: cubic e} is proved in Lemma \ref{lem:cubic-holds-in-bqt}. Before moving on to prove Theorem \ref{thm: E plus}, we examine a few important consequences of it. 

\begin{corollary}
\label{cor: E minus}
The elements $f_m,\psi_m^{\pm} \in \DBqt^{\leq}$ satisfy relations \eqref{eq: quadratic f}, \eqref{eq: commutation f with psi}, \eqref{eq: cubic f}, \eqref{eq: psi commute}, and \eqref{eq: psi 0 invertble}.
\end{corollary}

\begin{proof}
This follows from Theorem \ref{thm: E plus} by applying the anti-involution $\Theta$. By Lemmas \ref{lem: theta e f EHA} and \ref{lem: theta e f} the involutions $\Theta$ and $\Theta_{\mathcal{E}}$ agree.
\end{proof}

\begin{corollary}\label{cor: bqt to EHA}
There exist isomorphisms 
\begin{equation}
\label{eq: bqt to EHA}
\Eqt^{>}\simeq \EHA^{>},\ \Eqt^{\geq}\simeq \EHA^{\geq},\
\Eqt^{<}\simeq \EHA^{<},\ \Eqt^{\leq}\simeq \EHA^{\leq}.\
\end{equation}
and 
$$
\Eqt\simeq \EHA.
$$
\end{corollary}

\begin{proof}
Since all relations of $\EHA^{>}$ are satisfied in $\Eqt^{>}$, we get a surjective homomorphism $\EHA^{>}\to \Eqt^{>}$. On the other hand, if $V=\bigoplus_{k=0}^{\infty}V_k$ is the polynomial representation of $\Bqt$ then $V_0$ is the polynomial representation of $\Eqt^{>}$ and by Lemma \ref{lem: polynomial reps agree} we get a commutative diagram
$$
\begin{tikzcd}
\EHA^{>} \arrow{r} \arrow{dr} &  \Eqt^{>} \arrow{d}\\
 & \End(V_0)
\end{tikzcd}
$$
But the polynomial representation is faithful for $\EHA$ by Theorem \ref{thm: poly EHA faithful}, so the maps $\EHA^{>}\to \End(V_0)$ and $\EHA^{>}\to \Eqt^{>}$ are both injective.

The proof of the second isomorphism is similar, and the third and fourth ones follow by applying $\Theta$.
This implies that we have a surjective homomorphism $\EHA\to \Eqt$. By applying Theorem \ref{thm: poly EHA faithful} again, we conclude $
\Eqt\simeq \EHA.
$
\end{proof}

\subsection{Quadratic Relations}
\label{sec: quadratic check}

In this section we check the relations \eqref{eq: quadratic e}.

\begin{lemma}
\label{lem: level 2 zero}
Suppose that $a(z_1,z_2)$ is a symmetric Laurent polynomial in $z_1,z_2$. Then 
$$
d_{-}^2(z_1-qz_2)a(z_1,z_2)d_{+}^2\e_0=0.
$$
\end{lemma}

\begin{proof}
Recall that $a(z_1,z_2)$ is central in the affine Hecke algebra $\AH_2$. 
We have $q^{-1}z_1=T_1^{-1}z_2T_1^{-1}$, so
\begin{align*}
d_{-}^2(z_1-qz_2)a(z_1,z_2)d_{+}^2\e_0&=qd_{-}^2T_1^{-1}z_2T_1^{-1}a(z_1,z_2)d_+^2\e_0-qd_{-}^2z_2a(z_1,z_2)d_{+}^2\e_0
\\
&=qd_{-}^2T_1^{-1}z_2a(z_1,z_2)T_1^{-1}d_+^2\e_0-qd_{-}^2z_2a(z_1,z_2)d_{+}^2\e_0.
\end{align*}
Since $d_{-}^2T_1\e_2=d_{-}^2\e_2,T_1d_{+}^2=d_{+}^2$, we can delete both $T_1^{-1}$ and the terms will cancel out.
\end{proof}

\begin{lemma}
\label{lem: ee to level 2}
We have 
$$
e_{k+1}e_m-te_{k}e_{m+1}=d_{-}^2(q^{-1}z_1-tz_2)z_1^kz_2^md_{+}^2\e_0.
$$
\end{lemma}

\begin{proof}
Recall the relation \eqref{eq:qphi} which for $k=1$ has the form
$$
z_1(qd_+d_{-}-d_{-}d_{+})\e_1=qt(d_{+}d_{-}-d_{-}d_{+})z_1\e_1.
$$
We can divide both sides by $q$ and rewrite this as
\begin{equation}
\label{eq: d+ d- commutation level 1}
q^{-1}z_1d_{-}d_{+}\e_1-td_{-}d_{+}z_1\e_1=z_1d_{+}d_{-}\e_1-td_{+}d_{-}z_1\e_1.
\end{equation}
Now we multiply both sides by $d_{-}z_1^k$ on the left and $z_1^{m}d_{+}\e_0$ on the right:
$$
q^{-1}d_{-}z_1^kz_1d_{-}d_{+}z_1^{m}d_{+}\e_0-td_{-}z_1^kd_{-}d_{+}z_1z_1^{m}d_{+}\e_0=d_{-}z_1^kz_1d_{+}d_{-}z_1^{m}d_{+}\e_0-td_{-}z_1^kd_{+}d_{-}z_1z_1^{m}d_{+}\e_0.
$$
The right hand side equals $e_{k+1}e_m-te_ke_{m+1}$, while in the left hand side we can use the identities $z_1d_{-}=d_{-}z_1,d_{+}z_1=z_{2}d_{+}$ and obtain
$$
q^{-1}d_{-}^2z_1^{k+1}z_2^md_{+}^2\e_0-td_{-}^2z_1^kz_2^{m+1}d_{+}^2\e_0=
d_{-}^2(q^{-1}z_1-tz_2)z_1^kz_2^md_{+}^2\e_0.
$$
\end{proof}

\begin{lemma}\label{lem:quadratic-holds-in-bqt}
Relation  \eqref{eq: quadratic e} holds in $\Bqt$.
\end{lemma}

\begin{proof}
We can unpack equation \eqref{eq: quadratic e} as follows:
\begin{multline}
\label{eq: e quadratic expanded}
e_{n+3}e_m-\sigma_1e_{n+2}e_{m+1}+\sigma_2e_{n+1}e_{m+2}-e_ne_{m+3}=\\
-e_{m+3}e_n+\sigma_1e_{m+2}e_{n+1}-\sigma_2e_{m+1}e_{n+2}+e_me_{n+3}.
\end{multline} 
To prove it, we regroup the terms in the left hand side:
$$
e_{n+3}e_m-\sigma_1e_{n+2}e_{m+1}+\sigma_2e_{n+1}e_{m+2}-e_ne_{m+3}=
$$
$$
(e_{n+3}e_m-te_{n+2}e_{m+1})-(q+(qt)^{-1})(e_{n+2}e_{m+1}-te_{n+1}e_{m+2})+t^{-1}(e_{n+1}e_{m+2}-te_ne_{m+3}).
$$
By Lemma \ref{lem: ee to level 2} this equals
$$
d_{-}^2\left[z_1^2-(q+(qt)^{-1})z_1z_2+t^{-1}z_2^2\right](q^{-1}z_1-tz_2)z_1^nz_2^md_{+}^2\e_0=
$$
$$
td_{-}^2(z_1-qz_2)(z_1-(qt)^{-1}z_2)(z_2-(qt)^{-1}z_1)z_1^nz_2^md_{+}^2\e_0.
$$
To get the right hand side of \eqref{eq: e quadratic expanded}, we simply swap $m$ and $n$ and change sign.
Therefore to prove \eqref{eq: e quadratic expanded} we need to check that
\[
d_{-}^2(z_1-qz_2)(z_1-(qt)^{-1}z_2)(z_2-(qt)^{-1}z_1)z_1^nz_2^md_{+}^2\e_0=
\]
\[
-d_{-}^2(z_1-qz_2)(z_1-(qt)^{-1}z_2)(z_2-(qt)^{-1}z_1)z_1^mz_2^nd_{+}^2\e_0
\]
or, equivalently,
$$
d_{-}^2(z_1-qz_2)(z_1-(qt)^{-1}z_2)(z_2-(qt)^{-1}z_1)(z_1^nz_2^m+z_1^mz_2^n)d_{+}^2\e_0=0.
$$
This follows from Lemma \ref{lem: level 2 zero} since the polynomial $(z_1-(qt)^{-1}z_2)(z_2-(qt)^{-1}z_1)(z_1^nz_2^m+z_1^mz_2^n)$ is symmetric in $z_1,z_2$.
\end{proof}

\subsection{Cubic Relations}\label{sec: cubic check}
In this section, we show that the cubic relation \eqref{eq: cubic e} holds in $\Bqt$. We start with the following lemma.  
\begin{lemma}
\label{lem: alpha beta decomposition}
For all $m_1,m_2\in \Z$ we can write
$$
z_1^{m_1}z_2^{m_2}=(q^{-1}z_1-tz_2)\alpha+(z_1-qz_2)\beta
$$
for some Laurent polynomial $\alpha$ and some symmetric Laurent polynomial $\beta$.
\end{lemma}

\begin{proof}
Dividing by a power of $(z_1z_2)$ reduces us either to the case $m_1\ge 0,m_2=0$ or $m_1=0,m_2\ge 0$.

For $m_1=m_2=0$ we have an identity
$$
(q^{-1}z_1-tz_2)(z_1^{-1}+tz_2^{-1})-q^{-1}t(z_1-qz_2)(z_1^{-1}+z_2^{-1})=(t+q^{-1})(1-t),$$
so
\begin{equation}
\label{eq: alpha beta decompose 1}
1=\frac{1}{(t+q^{-1})(1-t)}\left[(q^{-1}z_1-tz_2)(z_1^{-1}+tz_2^{-1})-q^{-1}t(z_1-qz_2)(z_1^{-1}+z_2^{-1})\right]
\end{equation}
For $m_1=1,m_2=0$ we have an identity
$$
z_1=\frac{1}{q^{-1}(1-t)}\left[(q^{-1}z_1-tz_2)-q^{-1}t(z_1-qz_2)\right],
$$
and for $m_1=0,m_2=1$ we get
$$
z_2=\frac{1}{(1-t)}\left[(q^{-1}z_1-tz_2)-q^{-1}(z_1-qz_2)\right].
$$
The result follows from induction by observing that
$$
z_1^{m_1+1}=z_1^{m_1}(z_1+z_2)-z_1^{m_1-1}(z_1z_2),\quad 
z_2^{m_2+1}=z_2^{m_2}(z_1+z_2)-z_2^{m_2-1}(z_1z_2).
$$
\end{proof}

\begin{corollary}
We have $d_{-}^2z_1^{m_1}z_2^{m_2}d_+^2\in \Eqt^{>}$.
\end{corollary}

\begin{proof}
 We can use Lemma \ref{lem: alpha beta decomposition} and rewrite this as
$$
d_{-}^2z_1^{m_1}z_2^{m_2}d_+^2=d_{-}^2\left[(q^{-1}z_1-tz_2)\alpha+(z_1-qz_2)\beta\right]d_{+}^2.
$$
where $\alpha$ is a Laurent polynomial and $\beta$ is a symmetric Laurent polynomial.
By Lemma \ref{lem: ee to level 2} $d_{-}^2(q^{-1}z_1-tz_2)\alpha d_{+}^2$ is in $\Eqt^{>}$ while
by Lemma \ref{lem: level 2 zero} we get 
$d_{-}^2(z_1-qz_2)\beta d_{+}^2=0.$
\end{proof}

\begin{example}
\label{eq: Y 00}
By Lemma \ref{lem: ee to level 2} we get
$$
e_1e_{-1}-te_0^2=d_{-}^2(q^{-1}z_1-tz_2)z_2^{-1}d_{+}^2\e_0,
$$
$$
e_0^2-te_{-1}e_{1}=d_{-}^2(q^{-1}z_1-tz_2)z_1^{-1}d_{+}^2\e_0.
$$
and by \eqref{eq: alpha beta decompose 1} and Lemma \ref{lem: level 2 zero} we get
\begin{align*}
d_{-}^2d_{+}^2\e_0&=\frac{1}{(t+q^{-1})(1-t)}d_{-}^2(q^{-1}z_1-tz_2)(z_1^{-1}+tz_2^{-1})d_{+}^2\e_0
\\
&=\frac{1}{(t+q^{-1})(1-t)}\left[e_0^2-te_{-1}e_{1}+t(e_1e_{-1}-te_0^2)\right]
\\
&=\frac{1-t^2}{(t+q^{-1})(1-t)}e_0^2+\frac{t}{(t+q^{-1})(1-t)}[e_1,e_{-1}].
\end{align*}
\end{example}

\begin{lemma}
\label{lem: cubic rewritten}
We have the relation
$
d_{-}^2d_{+}^2d_{-}d_{+}\e_0=d_{-}d_{+}d_{-}^2d_{+}^2\e_0
$
in $\Bqt$.
\end{lemma}

\begin{proof}
We simplify $d_{-}^2\varphi d_{+}^2\e_0$ in two different ways:
\begin{align*}
d_{-}^2\varphi d_{+}^2\e_0&=d_{-}^2(\varphi d_{+})d_{+}\e_0=
qd_{-}^2T_1^{-1}d_{+}\varphi d_{+}\e_0=qd_{-}^2d_{+}\varphi d_{+}\e_0
\\
&=\frac{q}{q-1}(d_{-}^2d_{+}^2d_{-}d_{+}-d_{-}^2d_{+}d_{-}d_{+}^2)\e_0
\end{align*}
and
\begin{align*}
d_{-}^2\varphi d_{+}^2\e_0&=d_{-}(d_{-}\varphi)d_{+}^2\e_0=qd_{-}\varphi d_{-}T_1^{-1}d_{+}^2\e_0=qd_{-}\varphi d_{-}d_{+}^2\e_0
\\
&=\frac{q}{q-1}(d_{-}d_{+}d_{-}^2d_{+}^2-d_{-}^2d_{+}d_{-}d_{+}^2)\e_0.
\end{align*}
By comparing these, we get the desired relation.
\end{proof}

\begin{lemma}
\label{lem:cubic-holds-in-bqt}
The cubic relation \eqref{eq: cubic e} holds in $\Bqt$.
\end{lemma}

\begin{proof}
By Lemma \ref{lem: cubic rewritten} the operator $d_{-}^2d_{+}^2\e_0$ commutes with $e_0=d_{-}d_{+}\e_0$.
By Example \ref{eq: Y 00} this implies that  $[e_1,e_{-1}]$ commutes with $e_0$ as well, and $[e_0,[e_1,e_{-1}]]=0$.
\end{proof}

\begin{remark}
A different proof of Lemma \ref{lem: cubic rewritten} (and hence of the cubic relation \eqref{eq: cubic e}) follows from the work of Novarini \cite{Novarini}. Indeed, \cite[Proposition 2.5.6]{Novarini} implies that the operators $L_a=\e_0d_{-}\varphi^{a}d_{+}\e_0,\ a\in \Z_{\ge 0}$ pairwise commute. On the other hand, $L_0=\e_0d_{-}d_{+}\e_0$ and $$L_1=\e_0d_{-}\varphi d_{+}\e_0=\frac{1}{q-1}(L_0^2-\e_0d_{-}^2d_{+}^2\e_0),$$
so $L_0$ commutes with $\e_0d_{-}^2d_{+}^2\e_0$.
See also Remark \ref{rem: from Y to cubic}.
\end{remark}

\subsection{Commutation of $e(\z)$ and $\psi^{\pm}(\z)$.}

In this subsection we verify the relations \eqref{eq: commutation e with psi}.

\begin{lemma}
\label{lem: delta e commutation new}
We have 
\begin{equation}
\label{eq: delta e commutation}
[\Delta_{p_m},e_k]=e_{k+m}, [\Delta^*_{p_m},e_k]=e_{k-m}
\end{equation}
and
\begin{equation}
\label{eq: delta f commutation}
[\Delta_{p_m},f_k]=-f_{k+m}, [\Delta^*_{p_m},f_k]=-f_{k-m}.
\end{equation}
\end{lemma}

\begin{proof}
Recall that $[\Delta_{p_m},d_{-}]=[\Delta_{p_m},z_1]=0$ and $[\Delta_{p_m},d_{+}]=z_1^md_+$.
This implies
$$
[\Delta_{p_m},e_k]=[\Delta_{p_m},d_{-}z_1^kd_+]=d_{-}z_1^k[\Delta_{p_m},d_{+}]=d_{-}z_1^kz_1^md_+=e_{k+m}.
$$
The proof for $\Delta^*_{p_m}$ is similar, and the relations \eqref{eq: delta f commutation} follow by applying $\Theta$.
\end{proof}

\begin{lemma}
\label{lem: e psi commutation abstract}
The relations \eqref{eq: commutation e with psi}  
are satisfied in $\DBqt^{\geq}$.
\end{lemma}

\begin{proof}
We introduce auxiliary formal series
$$
X^+(\z)=-\sum_{m=1}^{\infty}\frac{\z^{-m}}{m}\Delta_{p_m}(1-q^m)(1-t^m)(1-(qt)^{-m})
$$
and 
$$
X^-(\z)=\sum_{m=1}^{\infty}\frac{\z^m}{m}\Delta^*_{p_m}(1-q^m)(1-t^m)(1-(qt)^{-m}).
$$
such that $\psi^{\pm}(\z)=\left(-\frac{1-q^{-1}t^{-1}\z^{-1}}{1-\z^{-1}}\right)\exp[X^{\pm}(\z)]$.  
We claim that  
\begin{equation}
\label{eq: X plus and e}
[X^+(\z),e(\w)]=-\sum_{m=1}^{\infty}\frac{\z^{-m}}{m}\w^m(1-q^m)(1-t^m)(1-(qt)^{-m})e(\w)
\end{equation}
and 
\begin{equation}
\label{eq: X minus and e}
[X^-(\z),e(\w)]=\sum_{m=1}^{\infty}\frac{\z^{m}}{m}\w^{-m}(1-q^m)(1-t^m)(1-(qt)^{-m})e(\w).
\end{equation}
Indeed, we have $[\Delta_{p_m},e_k]=e_{k+m}$ by Lemma \ref{lem: delta e commutation new} which can be written as
$$
[\Delta_{p_m},e(\w)]=\sum_{k\in \Z} [\Delta_{p_m},e_k]\w^{-k}=\sum_{k\in \Z} \w^{-k}e_{k+m}=\w^{m}e(\w).
$$
Similarly, $[\Delta^*_{p_m},e(\w)]=\w^{-m}e(\w)$ 
and the equations \eqref{eq: X plus and e}-\eqref{eq: X minus and e} follow.

The generating function $\psi^{+}(\z)$ is a power series in $\z^{-1}$, and $\psi^{-}(\z)$ is a power series in $\z$. Both power series are invertible, and we get
\begin{align*}
\Ad_{\psi^{+}(\z)}e(\w)&=\Ad_{\exp(X^{+}(\z))}e(\w)=\exp[\ad X^{+}(\z)]e(\w)
\\
&=\exp\left[-\sum_{m=1}^{\infty}\frac{\z^{-m}}{m}\w^m(1-q^m)(1-t^m)(1-(qt)^{-m})\right]e(\w)=
-\frac{g(\w,\z)}{g(\z,\w)}e(\w)
\end{align*}
by \eqref{eq: X plus and e} and Lemma \ref{lem: exp identity}.

Similarly, by \eqref{eq: X minus and e} and swapping $\z$ and $\w$ in Lemma \ref{lem: exp identity} we get 
\begin{align*}
\Ad_{\psi^{-}(\z)}e(\w)&=\Ad_{\exp(X^{-}(\z))}e(\w)=\exp[\ad X^{-}(\z)]e(\w)
\\
&=\exp\left[\sum_{m=1}^{\infty}\frac{\z^{m}}{m}\w^{-m}(1-q^m)(1-t^m)(1-(qt)^{-m})\right]e(\w)=
-\frac{g(\w,\z)}{g(\z,\w)}e(\w).
\end{align*}
\end{proof}

\begin{remark}
Some references define $\EHA$ using additional generators similar to our $\Delta_{p_m},\Delta^*_{p_m}$. For example, \cite[Definition 3.2]{neguct2023r} uses the generators $h_m,m\in \Z\setminus \{0\}$ which can be translated to our conventions  as
$$
h_m=\begin{cases}
\Delta_{p_m} & \text{if}\ m>0\\
-\Delta^*_{p_{|m|}} & \text{if}\ m<0.
\end{cases}
$$
Then \cite[eq. (3.3)-(3.4)]{neguct2023r} agrees with our \eqref{eq: delta e commutation} and \eqref{eq: delta f commutation}, while the definition of $\psi^{\pm}(\z)$ in terms of $h_m$ in \cite[eq. (3.7)]{neguct2023r} agrees with our \eqref{eq: def psi plus}-\eqref{eq: def psi minus}.
\end{remark}

\section{The spherical subalgebra}
\label{sec: spherical}

 We have seen that the algebra $\Eqt^{>} \subseteq \e_0\Bqt\e_0$ (resp. $\Eqt^{\geq} \subseteq \e_0\DBqt^{\geq}\e_0$ and $\Eqt \subseteq \e_0\DBqt\e_0$) is isomorphic to the positive half $\EHA^{>}$ (resp. non-negative half $\EHA^{\geq}$ and entire $\EHA$) of the elliptic Hall algebra. In this section, we will show that, in fact, every element of $\e_0\Bqt\e_0$ (resp. $\e_0\DBqt^{\geq}\e_0$ and $\e_0\DBqt\e_0$) can be written as a polynomial in the elements $e_k$ (resp. $\{e_k,\psi^{\pm}_{m}\}$ and $\{e_k, f_k, \psi^{\pm}_{m}\}$). The following is the main result of this section.

\begin{theorem}\label{thm:main-spherical}
We have the following equalities:
\[
\Eqt^{>} = \e_0\Bqt\e_0, \qquad \Eqt^{\geq} = \e_0\DBqt^{\geq}\e_0, \qquad \Eqt = \e_0\DBqt\e_0. 
\]
and consequently, by Corollary \ref{cor: bqt to EHA}, the following isomorphisms:
\[
\EHA^{>} \cong \e_0\Bqt\e_0, \qquad \EHA^{\geq} \cong \e_0\DBqt^{\geq}\e_0, \qquad \EHA \cong \e_0\DBqt\e_0. 
\]
\end{theorem}

\begin{remark}
Thanks to Theorem \ref{thm:main-spherical}, we have functors
\[
G^{>}: \Bqt\modd \to \EHA^{>}\modd, \quad G^{\geq}: \DBqt^{\geq}\modd \to \EHA^{\geq}\modd, \quad G: \DBqt\modd \to \EHA\modd,
\]
with their respective left adjoints
\[
F^{>}: \EHA^{>}\modd \to \Bqt\modd, \quad F^{\geq}: \EHA^{\geq}\modd \to \DBqt^{\geq}\modd, \quad F: \EHA\modd \to \DBqt\modd. 
\]
We note, however, that none of the functors are equivalences: in \cite[Remark 4.2]{calibrated} we constructed a nonzero representation of $\DBqt^{\geq}$ that is annihilated by the element $\e_0$. This representation can in turn be used to construct a nonzero $\DBqt$-representation that is annihilated by $\e_0$. 
\end{remark}

Let us sketch the strategy to prove Theorem \ref{thm:main-spherical}. First, we observe that it is enough to show the equality $\Eqt^{>} = \e_0\Bqt\e_0$. Indeed, if this is true then applying the involution $\adj$ from Lemma \ref{lem:antiinvolution-double}, we will have also that $\Eqt^{<} = \e_0(\DBqt^{<})\e_0$. So $\e_0\DBqt\e_0$ is generated by the elements $e_k, f_k$ and $\e_0\Delta_{p_m}\e_0, \e_0\Delta^{\ast}_{p_m}\e_0$, which implies that $\Eqt = \e_0\DBqt\e_0$. In turn, by restricting to the non-negative part of these algebras, this implies that $\Eqt^{\geq} = \e_0\DBqt^{\geq}\e_0$.

The crux of the argument is then to show that every element in $\e_0\Bqt\e_0$ can be written as a polynomial in $e_k$. First, in Section \ref{sec: special} we find a quite large generating set $\{Y_{m_1, \dots, m_n} \mid n > 0, m_1, \dots, m_n \in \Z\}$ for $\e_0\Bqt\e_0$. In Sections \ref{sec:Yrelns-statement} and \ref{sec: Y proofs} we will show that the $Y$ elements satisfy, among themselves and with the $e$-elements, the relations (49) and (50) in \cite{GNtrace1}. By \cite[Theorem 2.13]{GNtrace1}, this implies that every $Y$-element is generated by the $e$-elements, and this will finish the proof of Theorem \ref{thm:main-spherical}. 

\begin{remark}
The elliptic Hall algebra $\EHA$ is equipped with a coproduct and an antipode \cite{EHA} which make it a Hopf algebra. As a consequence, the category of $\EHA$-modules is equipped with a monoidal structure. Both the coproduct and antipode preserve the non-negative part $\EHA^{\geq}$, so that in particular the category of $\EHA^{\geq}$-modules also admits a monoidal structure.

In \cite{calibrated} we have defined a ``generically" monoidal structure on the category of calibrated representations of $\Bqt^{\ext}$. This structure does not come from a coproduct on $\Bqt^{\ext}$, as the  product of two representations is usually larger than the tensor product of the underlying vector spaces:
$$
V_k(E_1\times E_2)=\bigoplus_{m+n=k}\mathrm{Ind}_{\AH_m\otimes \AH_n}^{\AH_{m+n}}V_m(E_1)\otimes V_n(E_2).
$$
See \cite[Section 6]{calibrated} for more details.
The definition of monoidal structure also depends on two rational functions $\psi_1$ and $\psi_2$.

However, $V_0(E_1\times E_2)=V_0(E_1)\otimes V_0(E_2)$, therefore for (calibrated) representations of $\e_0\Bqt^{\ext}\e_0$ the monoidal structure does agree with the tensor product of the underlying vector spaces. One can check by comparing \cite[Lemma 4.11]{EHA} with \cite[Theorem 6.14]{calibrated} that the monoidal structures for $\EHA^{\geq}$ and $\e_0\Bqt^{\ext}\e_0$ agree for an appropriate choice of functions $\psi_1$ and $\psi_2$.
\end{remark}

\subsection{Special Elements}
\label{sec: special}

In this section, we 
find a (quite large) generating set for the algebra $\e_0\Bqt\e_0$. Recall that we have the element $\varphi = \frac{1}{q-1}[d_{+}, d_{-}]$, that we may think of as an extra generator of $\Bqt$.

\begin{definition}
We call an element of $\e_0\Bqt\e_k$ \newword{special} if it can be written as a product of $d_{-},\varphi$ and $z_i$ (but not $T_i$ or $d_{+}$). 

We call an element of $\e_0\Bqt\e_k$ \newword{factorizable} if it can be written as the product $XY$ for $X\in \e_0\Bqt\e_0$ not a unit and $Y\in \e_0\Bqt\e_k$. 
\end{definition}

Note that the special elements are only defined for $k\ge 1$.
We also define the following elements: 
\begin{equation}
\label{eq: def A}
A_{m_1,\ldots,m_n}:=\e_0d_{-}(z_1^{m_1}\varphi\cdots\varphi z_1^{m_{n-1}}\varphi z_1^{m_n})\e_1\in \e_0\Bqt\e_1
\end{equation}
and 
\begin{equation}
\label{eq: def Y}
Y_{m_1,\ldots,m_n}:=\e_0d_{-}(z_1^{m_1}\varphi\cdots\varphi z_1^{m_{n-1}}\varphi z_1^{m_n})d_{+}\e_0=A_{m_1,\ldots,m_n}d_{+}\in \e_0\Bqt\e_0.
\end{equation}
Here $m_i$ are arbitrary integers, and $n\ge 1$. Note that $A_0=\e_0d_{-}\e_1$ and $Y_{m}=e_m$. The following is also clear from the definitions.

\begin{proposition}
The elements $A_{m_1,\ldots,m_n}\in \e_0\Bqt\e_1$ are special, and any special element of $\e_0\Bqt\e_1$ equals $A_{m_1,\ldots,m_n}$ for some $m_i\in \Z$.
\end{proposition}

We recall some useful relations for $\varphi$.

\begin{lemma}
\label{lem: phi relations}
a) For $1<i\le k$ we have $z_i\varphi=\varphi z_{i-1}$ and
$T_i\varphi=\varphi T_{i-1}$.

b) We have $\varphi^2 T_{k-1}=T_1\varphi^2$.
\end{lemma}
\begin{proof}
a) We have $z_id_{+}d_{-}=d_{+}z_{i-1}d_{-}=d_{+}d_{-}z_{i-1}$ and $z_id_{-}d_{+}=d_{-}z_id_{+}=d_{-}d_{+}z_{i-1}$. The second equation is proved similarly.

b) This is proved in \cite[Lemma 3.1.11]{CGM}.
\end{proof}

Recall that, at the vertex $k$ of the quiver $Q_\Z$ (see Section \ref{sec: def DBqt}), the algebra generated by $z_1, \dots, z_k$ and $T_1, \dots, T_{k-1}$ is the affine Hecke algebra $\AH_k$. 
For $w\in S_k$ with a reduced expression $w=s_{i_1}\cdots s_{i_\ell}$ we consider the corresponding element $T_w=T_{i_1}\cdots T_{i_\ell}\in \AH_k$. It is well known that $\AH_k$ has a basis 
\begin{equation}
\label{eq: affine Hecke basis}
T_wz_1^{a_1}\cdots z_k^{a_k},\ w\in S_k,\ a_i\in \Z.
\end{equation}
In particular, any expression $f(z_1,\ldots,z_k)T_v$ can be re-expanded in the basis \eqref{eq: affine Hecke basis} which we refer to as ``pushing $f(z_1,\ldots,z_k)$ through $T_{v}$."

\begin{lemma}
\label{lem: special T}
Suppose $k\ge 1$ and $A\in \e_0\Bqt\e_k$ is special. Then for any $M\in \AH_k$ we can write the product $AM$ as a linear combination of special elements.
\end{lemma}

\begin{proof}
If $k=1$ then $M$ is a polynomial in $z_1$ and there is nothing to prove. So we will assume $k\ge 2$.

First, we make several easy reductions. If $A$ ends with a monomial in $z_1,\ldots,z_k$ then we can absorb it into $M$, so we can assume that $A$ ends with $d_{-}$ or $\varphi$. 

Without loss of generality we can assume $M=T_w z_1^{a_1}\cdots z_k^{a_k}$ as in \eqref{eq: affine Hecke basis}.
Then $AM=AT_{w}z_1^{a_1}\cdots z_k^{a_k}$ and it is sufficient to consider the case $a_1=\ldots=a_k=0$. 
To sum up, it is sufficient to prove the statement for $M=T_i, 1\le i\le k-1$. 

We prove this by induction in the total number of $d_{-}$ and $\varphi$ in $A$. The base of induction is $A=d_{-}$ where $k=1$ and there is nothing to prove. 

For the inductive step, we have the following cases:

{\bf Case 1:} $A=A'd_{-}$ and $i<k-1$. Then
$AT_i=A'd_{-}T_{i}=A'T_id_{-}$ and the assumption of induction applies to $A'T_i$.

{\bf Case 2:} $A=A'\varphi$ and $i<k-1$. Then
$AT_i=A'\varphi T_{i}=A'T_{i+1}\varphi$ by Lemma \ref{lem: phi relations}
and the assumption of induction applies to $A'T_{i+1}$.

{\bf Case 3:} $A=A'd_{-}f(z_1,\ldots,z_{k-1})d_{-}$ and $i=k-1$. First, we can get rid of $f$ by writing 
$$
AT_{k-1}=A'd_{-}^2f(z_1,\ldots,z_{k-1})T_{k-1},
$$
pushing $f(z_1,\ldots,z_{k-1})$ through $T_{k-1}$ and arguing as above. 
Now
$$
A'd_{-}^2T_{k-1}=A'd_{-}^2.
$$

{\bf Case 4:} $A=A'\varphi f(z_1,\ldots,z_{k-1})d_{-}$ and $i=k-1$. Similarly to Case 3 we can assume $f(z_1,\ldots,z_{k-1})=1$. Now
$$
A'\varphi d_{-}T_{k-1}=q A'\varphi d_{-}T_{k-1}^{-1}+(1-q)A'\varphi d_{-}=A'd_{-}\varphi+(1-q)A'\varphi d_{-}.
$$

{\bf Case 5:} $A=A'd_{-} f(z_1,\ldots,z_{k})\varphi$ and $i=k-1$. We can again eliminate $f(z_1,\ldots,z_k)$ using the relations $d_{-}z_1=z_1d_{-}$  and $z_{j}\varphi=\varphi z_{j-1}$ for $j>1$. Now 
$$
A'd_{-}\varphi T_{k-1}=qA'\varphi d_{-}.
$$

{\bf Case 6:} $A=A'\varphi f(z_1,\ldots,z_k)\varphi$ and $i=k-1$.  We can again eliminate $f(z_1,\ldots,z_k)$ similalry to Case 5. Now
$$
A'\varphi^2T_{k-1}=A'T_1\varphi^2
$$
by Lemma \ref{lem: phi relations}(b) and the assumption of induction applies to $A'T_1$.
\end{proof}

\begin{lemma}
\label{lem: special d+}
Suppose $k\ge 2$ and $A\in \e_0\Bqt\e_k$ is special. Then  we can write the product $Ad_{+}$ as a linear combination of special and factorizable elements.
\end{lemma}

\begin{proof}
Again, we prove it by induction in the total number of $d_{-}$ and $\varphi$ in $A$. We have the following cases:

{\bf Case 1:} $A=A'd_{-}f(z_1,\ldots,z_{k}).$ Then 
$$
Ad_{+}=A'd_{-}f(z_1,\ldots,z_{k})d_{+}
$$
and we can eliminate $f(z_1,\ldots,z_{k})$ by using the relations $d_{-}z_1=z_1d_{-}$ and $z_jd_{+}=d_{+}z_{j-1}$ for $j>1$. 
Now
$$
A'd_{-}d_{+}=A'd_{+}d_{-}-(q-1)A'\varphi,
$$
where $A'\in \e_0\Bqt \e_{k-1}$. 
If $k>2$, we can apply the assumption of induction to $A'd_{+}$. If $k=2$ then $A'\in \e_0\Bqt \e_{1}$ and $A'd_{+}\in \e_0\Bqt \e_{0}$, so $A'd_{+}d_{-}$ is factorizable.

{\bf Case 2:} $A=A'\varphi f(z_1,\ldots,z_{k}).$ Then 
$$
Ad_{+}=A'\varphi f(z_1,\ldots,z_{k})d_{+}
$$
and we can eliminate $f(z_1,\ldots,z_{k})$ by using the relations $\varphi z_1=z_2\varphi$ and $z_jd_{+}=d_{+}z_{j-1}$ for $j>1$. Now
$$
A'\varphi d_{+}=qA'T_1^{-1}d_{+}\varphi.
$$
By Lemma \ref{lem: special T} we can write $A'T_1^{-1}=\sum A''_j$  where $A''_j$ are special, and then apply
the assumption of induction to $A''_jd_{+}$. 
\end{proof}

\begin{theorem}
\label{thm: special generate}
a) For $k>0$ any element of $\e_0\Bqt\e_k$ can be written as a linear combination of special and factorizable elements.

b) For $k=0$ any element of $\e_0\Bqt\e_0$ can be written as a linear combination of $Y_{m_1,\ldots,m_n}$ and factorizable elements.

c) The subalgebra $\e_0\Bqt\e_0$ is generated by the elements $Y_{m_1,\ldots,m_n}$.
\end{theorem}

\begin{proof}
Part (c) is a consequence of (b). We prove parts (a) and (b) for any monomial in $\e_0\Bqt$ by induction in the total number of $d_{+},d_{-}$ and $\varphi$. The base case is $\e_0d_{-}$ which is clearly special. 

For the step of induction, consider $A\in \e_0\Bqt\e_k$ which satisfies the conditions of the theorem.  If $k=0$ then $A$ can be only followed by $d_{-}$, and $Ad_{-}$ is factorizable. From now we assume that $k\ge 1$ and $A$ is special.

Clearly, $Ad_{-}$ and $A\varphi$ are special. If $M\in \AH_k$ then by Lemma \ref{lem: special T} we can write $AM$ as a linear combination of special elements. If $k\ge 2$ by Lemma \ref{lem: special d+} we can write $Ad_{+}$ as a linear combination of special and factorizable elements. 

Finally, if $k=1$ and $A\in \e_0\Bqt\e_1$ is special, then $A=A_{m_1,\ldots,m_n}$ for some $m_1,\ldots,m_n\in \Z$. Then
$$
Ad_{+}=A_{m_1,\ldots,m_n}d_{+}=Y_{m_1,\ldots,m_n}.
$$
\end{proof}

\subsection{Relations for $Y_{\underline{m}}$}\label{sec:Yrelns-statement}

In this section we prove that $\e_0\Bqt\e_0$ is generated by $e_k$. By Theorem \ref{thm: special generate}(c) it is sufficient to prove that for all $m_1,\ldots,m_n$ the element $Y_{m_1,\ldots,m_n}$ can be written as a polynomial in $e_k$.
For this, we need some relations for $Y_{\underline{m}}$.

\begin{theorem}
\label{thm: Y commutation relation} The following relations hold in $\Bqt$: 

\begin{equation}
\label{eq: Y relation 1}
Y_{m_1,\ldots,m_i,m_{i+1},\ldots,m_n}-qtY_{m_1,\ldots,m_i-1,m_{i+1}+1,\ldots,m_n}=-Y_{m_1,\ldots,m_i}Y_{m_{i+1},\ldots,m_n},
\end{equation}
\begin{equation}
\label{eq: Y relation 2}
[e_k,Y_{m_1,\ldots,m_n}]=(t-1)(q-1)\sum_{i=1}^{n}\begin{cases}
\sum_{a=1}^{k-m_i}Y_{m_1,\ldots,m_{i-1},k-a,m_i+a,m_{i+1},\ldots,m_n} & \text{if}\ k>m_i\\
0 & \text{if}\ k=m_i\\
-\sum_{a=1}^{m_i-k}Y_{m_1,\ldots,m_{i-1},m_i-a,k+a,m_{i+1},\ldots,m_n} & \text{if}\ k< m_i.
\end{cases}
\end{equation}
\end{theorem}

We prove Theorem \ref{thm: Y commutation relation} below in Section \ref{sec: Y proofs}.

\begin{example}
\label{ex: degree 2 example}
For $n=2$ we have
\begin{align*}
Y_{m_1,m_2}&=d_{-}z_1^{m_1}\varphi z_1^{m_2}d_{+}=\frac{1}{q-1}d_{-}z_1^{m_1}d_{+}d_{-}z_1^{m_2}d_{+}-\frac{1}{q-1}d_{-}z_1^{m_1}d_{-}d_{+}z_1^{m_2}d_{+}
\\
&=\frac{1}{q-1}e_{m_1}e_{m_2}-\frac{1}{q-1}d_{-}^2z_1^{m_1}z_2^{m_2}d_{+}^2.
\end{align*}
\end{example}

\begin{remark}
\label{rem: from Y to cubic}
We can compare the relations \eqref{eq: Y relation 1} and \eqref{eq: Y relation 2} with earlier computations: 
\begin{itemize}
\item The relation \eqref{eq: Y relation 1} implies
$$
Y_{m_1,m_2}-qtY_{m_1-1,m_2+1}+e_{m_1}e_{m_2}=0
$$
which is  equivalent to Lemma \ref{lem: ee to level 2}.

\item  By \eqref{eq: Y relation 2} we get
$$
[e_{1},e_{-1}]=(q-1)(t-1)(Y_{0,0}+Y_{-1,1})
$$
while 
$$
Y_{0,0}-qtY_{-1,1}=-e_0^2,\ Y_{-1,1}=q^{-1}t^{-1}Y_{0,0}+q^{-1}t^{-1}e_0^2.
$$
By combining these, we get
$$
[e_{1},e_{-1}]-q^{-1}t^{-1}(q-1)(t-1)e_0^2=(1+q^{-1}t^{-1})(q-1)(t-1)Y_{0,0}
$$
which is equivalent to \eqref{eq: Y 00}.
\item Finally, by \eqref{eq: Y relation 2} we get
$
[e_0,Y_{0,0}]=0
$
which is equivalent to the cubic relation \eqref{eq: cubic e}.
\end{itemize}
\end{remark}

\begin{theorem}
\label{thm: spherical generated by e}
We have $\e_0\Bqt\e_0=\Eqt^{>}$.
\end{theorem}

\begin{proof}
By definition we have $\Eqt^{>}\subset \e_0\Bqt\e_0$, so we need to prove the reverse inclusion.

By Theorem \ref{thm: special generate} any element  of $\e_0\Bqt\e_0$ can be written as a polynomial in $Y_{m_1,\ldots,m_n}$. 
On the other hand, by \cite[Proposition 2.13]{GNtrace1} 
the relations \eqref{eq: Y relation 1} and \eqref{eq: Y relation 2} imply that 
$Y_{m_1,\ldots,m_n}$ can be written as polynomials in $e_k$. 
\end{proof}

\begin{remark}
By \cite[Proposition 4.9]{negut2018hecke} the relations \eqref{eq: Y relation 1} and \eqref{eq: Y relation 2} imply the quadratic and cubic relations \eqref{eq: quadratic e},\eqref{eq: cubic e} for $e_k=Y_{k}$. So  Theorem \ref{thm: Y commutation relation} yields yet another proof of Theorem \ref{thm: E plus}.
\end{remark}

\begin{corollary}
We have $\e_0\Bqt^{-}\e_0=\Eqt^{<}$ and $\e_0\DBqt\e_0=\Eqt$.
\end{corollary}

\begin{proof}
By applying the involution $\Theta$ we conclude that $\e_0\Bqt^{-}\e_0=\Eqt^{<}$.

By definition, we have $\Eqt\subset \e_0\DBqt\e_0$. Any monomial in $\e_0\DBqt\e_0$ can be written as a product of monomials in $\e_0\Bqt\e_0$, $\e_0\Bqt^{-}\e_0$ and $\e_0\Delta_{p_m}\e_0,\e_0\Delta_{p_m}^*\e_0$. By the above, all these are contained in $\Eqt$, so $\e_0\DBqt\e_0\subset \Eqt$.
\end{proof}

We have proved Theorem \ref{thm:main-spherical} assuming Theorem \ref{thm: Y commutation relation}, whose proof is contained in the next section.

\begin{remark}
In \cite{BHMPS3} Blasiak et al. consider more general elements of $\EHA$ (in the form of shuffle algebra) parametrized by a tuple of integers $(m_1,\ldots,m_n)$ and a triple of Dyck paths $R_q,R_t,R_{qt}$. It would be interesting to find the analogues of such elements in $\DBqt$.
\end{remark}

\subsection{Proof of Theorem \ref{thm: Y commutation relation}}
\label{sec: Y proofs}
We will use the notation
$$
h_{m-1}(z_1,z_2)=\frac{z_1^m-z_2^m}{z_1-z_2},\ m\in \Z.
$$
Note that $h_{m-1}(z_1,z_2)$ is symmetric in $z_1,z_2$ and we have
$$
h_{m-1}(z_1,z_2)=\begin{cases}
z_1^{m-1}+\ldots+z_2^{m-1} & \mathrm{if}\ m>0\\
0 & \mathrm{if}\ m=0\\
-(z_1^{m}z_2^{-1}+\ldots+z_1^{-1}z_2^m)=-(z_1z_2)^{m}h_{-m-1}(z_1,z_2)& \mathrm{if}\ m<0.\\
\end{cases}
$$
Also, we have a recursion 
\begin{equation}
\label{eq: h recursion}
h_{m-1}(z_1,z_2)=z_2^{m-1}+z_1h_{m-2}(z_1,z_2).
\end{equation}

\begin{lemma}
\label{lem: z through T inverse}
We have 
\begin{equation}
\label{eq: z_1 through T inverse}
z_1^aT_1^{-1}=T_1^{-1}z_2^a-q^{-1}(1-q)h_{a-1}(z_1,z_2)z_1
\end{equation}
and
\begin{equation}
\label{eq: z_2 through T inverse}
z_2^aT_1^{-1}=T_1^{-1}z_1^a+q^{-1}(1-q)h_{a-1}(z_1,z_2)z_1.
\end{equation}
\end{lemma}

\begin{proof}
We have
$z_1=qT_1^{-1}z_2T_1^{-1}$, so 
$$
z_2T_1^{-1}=q^{-1}T_1z_1=q^{-1}(1-q)z_1+T_1^{-1}z_1
$$
where we used the quadratic relation
$T_1=(1-q)+qT_1^{-1}.$

If $a=0$ then \eqref{eq: z_2 through T inverse} is clear. For $a\neq 0$ we prove it by induction by observing
$$
z_2^aT_1^{-1}=z_2^{a-1}(q^{-1}(1-q)z_1+T_1^{-1}z_1)=
q^{-1}(1-q)z_2^{a-1}z_1+(z_2^{a-1}T_1^{-1})z_1
$$
and using \eqref{eq: h recursion}.
Finally, \eqref{eq: z_1 through T inverse} follows from \eqref{eq: z_2 through T inverse} since 
$$
(z_1^a+z_2^a)T_1^{-1}=T_1^{-1}(z_1^a+z_2^a).
$$
\end{proof}

\begin{lemma}
\label{lem: push T left}
We have 
$$
\varphi d_{-}T_1^{-1}z_1^{k}d_{+}\e_1=d_{-}T_1^{-1}z_1^{k}d_{+}\varphi\e_1+q^{-1}(1-q)d_{-}h_{k-1}(z_1,z_2)z_1d_{+}\varphi\e_1.
$$
\end{lemma}

\begin{proof}
We have
\begin{align*}
\varphi d_{-}T_1^{-1}z_1^{k}d_{+}\e_1&=q^{-1}d_{-}\varphi z_1^{k}d_{+}\e_1=
q^{-1}d_{-}z_2^{k}\varphi d_{+}\e_1=
d_{-}z_2^{k}T_1^{-1}d_{+}\varphi\e_1
\\
&=d_{-}T_1^{-1}z_1^{k}d_{+}\varphi\e_1+q^{-1}(1-q)d_{-}h_{k-1}(z_1,z_2)z_1d_{+}\varphi\e_1.
\end{align*}
Here we apply the first relation in \eqref{eq:phi}, then Lemma \ref{lem: phi relations}(a), followed by the second relation in \eqref{eq:phi}, and finally Lemma \ref{lem: z through T inverse}.
\end{proof}

\begin{corollary}
\label{cor: push T left}
We have 
$$
z_1^{m}\varphi \left[d_{-}T_1^{-1}z_1^{k}d_{+}\right]\e_1=\left[d_{-}T_1^{-1}z_1^{k}d_{+}\right]z_1^{m}\varphi\e_1-q^{-1}(1-q)d_{-}(z_1z_2)^kh_{m-k-1}(z_1,z_2)z_1d_{+}\varphi\e_1.
$$
\end{corollary}

\begin{proof}
By Lemma \ref{lem: push T left} we get
$$
z_1^{m}\varphi d_{-}T_1^{-1}z_1^ad_{+}\e_1=z_1^md_{-}T_1^{-1}z_1^kd_{+}\varphi\e_1+q^{-1}(1-q)z_1^md_{-}h_{k-1}(z_1,z_2)z_1d_{+}\varphi\e_1.
$$
We can rewrite the terms as
\begin{align*}
z_1^md_{-}T_1^{-1}z_1^{k}d_{+}\varphi\e_1&=
d_{-}z_1^mT_1^{-1}z_1^{k}d_{+}\varphi\e_1
\\
&=d_{-}T_1^{-1}z_2^mz_1^{k}d_{+}\varphi\e_1-q^{-1}(1-q)d_{-}h_{m-1}(z_1,z_2)z_1^{k}d_{+}\varphi\e_1
\\
&=d_{-}T_1^{-1}z_1^{k}d_{+}z_1^m\varphi\e_1-q^{-1}(1-q)d_{-}h_{m-1}(z_1,z_2)z_1^{k}d_{+}\varphi
\e_1
\end{align*}
and 
$$
q^{-1}(1-q)z_1^md_{-}h_{k-1}(z_1,z_2)z_1d_{+}\varphi\e_1=
q^{-1}(1-q)d_{-}z_1^mh_{k-1}(z_1,z_2)z_1d_{+}\varphi\e_1.
$$
The result now follows from the identity
$$
z_1^mh_{k-1}(z_1,z_2)-z_1^{k}h_{m-1}(z_1,z_2)=-(z_1z_2)^{k}h_{m-k-1}(z_1,z_2).
$$
\end{proof}

\begin{lemma}
\label{lem: A relation 1}
We have 
\begin{equation}
\label{eq: A relation 1}
A_{m_1,\ldots,m_i,0}-qtA_{m_1,\ldots,m_i-1, 1}=-A_{m_1,\ldots,m_i}d_{+}d_{-}.
\end{equation}
\end{lemma}

\begin{proof}
Note that 
\[A_{m_1, \dots, m_i, 0} - qtA_{m_1, \dots, m_{i}-1, 1} = d_{-}z_1^{m_1}\varphi\cdots z_1^{m_{i-1}}\varphi z_1^{m_i -1}(z_1\varphi - qt\varphi z_1)\e_1.
\]
Now the equation \eqref{eq:qphi} can be rewritten as $z_1\varphi\e_1 + z_1d_+d_-\e_1 = qt\varphi z_1\e_1$, or equivalently $(z_1\varphi - qt\varphi z_1)\e_1 = -z_1d_{+}d_{-}\e_1$, and the result follows.
\end{proof}

\begin{lemma}
\label{lem: A relation 2}
We have 
\begin{multline}
\label{eq: A relation 2}
A_{m_1,\ldots,m_i,0} \left[d_{-}(z_1z_2)^kh_{m_i-k-1}(z_1,z_2)z_1d_{+}\right]\e_1=\\
qA_{m_1,\ldots,k,m_i}-qA_{m_1,\ldots,m_i,k}+
\begin{cases}
q(t-1)
\sum_{a=1}^{m_i-k}A_{m_1,\ldots,m_i-a,k+a} & \mathrm{if}\ m_i>k\\
0 & \mathrm{if}\ m_i=k\\
-q(t-1)
\sum_{a=1}^{k-m_i}A_{m_1,\ldots,k-a,m_i+a} & \mathrm{if}\ m_i<k.
\end{cases}
\end{multline}
\end{lemma}

\begin{proof} We divide in the three cases appearing in the statement of the lemma.  \\

{\bf Case 1:} $m_i> k$. In this case
$$
(z_1z_2)^kh_{m_i-k-1}(z_1,z_2)z_1=(z_1z_2)^k\left(\sum_{a=0}^{m_i-k-1}z_1^{m_i-k-1-a}z_2^{a}\right)z_1=\sum_{a=0}^{m_i-k-1}z_1^{m_i-a}z_2^{k+a},
$$
hence
\begin{align*}
d_{-}(z_1z_2)^kh_{m_i-k-1}(z_1,z_2)z_1d_{+}&=\sum_{a=0}^{m_i-k-1}
d_{-}z_1^{m_i-a}z_2^{k+a}d_{+}=
\sum_{a=0}^{m_i-k-1}
z_1^{m_i-a}d_{-}d_{+}z_1^{k+a}
\\
&=\sum_{a=0}^{m_i-k-1}z_1^{m_i-a}d_{+}d_{-}z_1^{k+a}-
(q-1)\sum_{a=0}^{m_i-k-1}z_1^{m_i-a}\varphi z_1^{k+a}.
\end{align*}
Therefore the left hand side of \eqref{eq: A relation 2} equals
$$
\sum_{a=0}^{m_i-k-1}A_{m_1,\ldots,m_{i-1},m_i-a}d_{+}d_{-}z_1^{k+a}-(q-1)\sum_{a=0}^{m_i-k-1}A_{m_1,\ldots,m_{i-1},m_i-a,k+a}.
$$
By \eqref{eq: A relation 1} we can rewrite it as
$$
\sum_{a=0}^{m_i-k-1}\left[-qA_{m_1,\ldots,m_i-a,k+a}+qtA_{m_1,\ldots,m_i-a-1,k+a+1}\right]=
$$
$$
qA_{m_1,\ldots,k,m_i}-qA_{m_1,\ldots,m_i,k}+q(t-1)
\sum_{a=1}^{m_i-k}A_{m_1,\ldots,m_i-a,k+a}.
$$

{\bf Case 2: $m_i=k$.} In this case $h_{m_i-k-1}=0$ and there is nothing to prove.

{\bf Case 3: $m_i<k$.} In this case
$$
(z_1z_2)^kh_{m_i-k-1}(z_1,z_2)z_1=
-(z_1z_2)^{m_i}h_{k-m_i-1}(z_1,z_2)z_1=-\sum_{a=0}^{k-m_i-1}z_1^{k-a}z_2^{m_i+a},
$$
and similarly to Case 1 the left hand side of \eqref{eq: A relation 2} equals
$$
-\sum_{a=0}^{k-m_i-1}A_{m_1,\ldots, k-a}d_{+}d_{-}z_1^{m_i+a}+(q-1)\sum_{a=0}^{k-m_i-1}A_{m_1,\ldots,k-a,m_i+a}=
$$
$$
-\sum_{a=0}^{k-m_i-1}\left[-qA_{m_1,\ldots,k-a,m_i+a }+qtA_{m_1,\ldots,k-a-1,m_i+a+1}\right]=
$$
$$
-qA_{m_1,\ldots,m_i,k}+qA_{m_1,\ldots,k,m_i}-q(t-1)
\sum_{a=1}^{k-m_i}A_{m_1,\ldots,k-a,m_i+a}.
$$
\end{proof}

Multiplying both sides of \eqref{eq: A relation 2} by $\varphi z_1^{m_{i+1}}\cdots \varphi z_1^{m_n}d_{+}\e_0$ on the right, we obtain the following result.

\begin{corollary}
\label{lem: Di to Y}
We have 
\begin{multline*}
\label{eq: G step}
d_{-}z_1^{m_1}\varphi
\cdots z_1^{m_{i-1}}\varphi \left[d_{-}(z_1z_2)^kh_{m_i-k-1}(z_1,z_2)z_1d_{+}\right]\varphi  z_1^{m_{i+1}}\cdots \varphi z_1^{m_n}d_{+}=\\
qY_{m_1,\ldots,k,m_i,\ldots,m_{n}}-qY_{m_1,\ldots,m_i,k,\ldots,m_{n}}+
\begin{cases}
q(t-1)
\sum_{a=1}^{m_i-k}Y_{m_1,\ldots,m_i-a,k+a,\ldots,m_{n}} & \mathrm{if}\ m_i>k\\
0 & \mathrm{if}\ m_i=k\\
-q(t-1)
\sum_{a=1}^{k-m_i}Y_{m_1,\ldots,k-a,m_i+a,\ldots,m_{n}} & \mathrm{if}\ m_i<k.
\end{cases}
\end{multline*}
\end{corollary}

\begin{lemma}
\label{lem: right end}
We have 
$$
z_1^md_{+}d_{-}z_1^kd_{+}\e_0=
$$
$$(q-1)z_1^m\varphi z_1^kd_{+}\e_0+d_{-}T_1^{-1}z_1^kd_{+}z_1^md_{+}\e_0-q^{-1}(1-q)d_{-}(z_1z_2)^kh_{m-k-1}(z_1,z_2)z_1d_{+}^2\e_0.
$$
\end{lemma}

\begin{proof}
We have 
$$
z_1^md_{+}d_{-}z_1^kd_{+}\e_0=(q-1)z_1^m\varphi z_1^kd_{+}\e_0+z_1^md_{-}d_{+}z_1^kd_{+}\e_0.
$$
The last term can be rewritten as
$$
z_1^md_{-}d_{+}z_1^kd_{+}\e_0=d_{-}z_1^mz_2^kd_{+}^2\e_0=d_{-}z_1^mz_2^{k}T_1^{-1}d_{+}^2\e_0=
$$
$$
d_{-}T_1^{-1}z_1^kz_2^md_{+}^2\e_0-q^{-1}(1-q)d_{-}(z_1z_2)^kh_{m-k-1}(z_1,z_2)z_1d_{+}^2\e_0.
$$
\end{proof}

We are ready to prove Theorem \ref{thm: Y commutation relation}.

\begin{proof}[Proof of Theorem \ref{thm: Y commutation relation}]

The first relation \eqref{eq: Y relation 1} follows from Lemma \ref{lem: A relation 1}.

To prove the second relation \eqref{eq: Y relation 2} we combine all the previous results in this section. Let $$
G_i=d_{-}z_1^{m_1}\varphi
\cdots z_1^{m_{i-1}}\varphi \left[d_{-}T_1^{-1}z_1^{k}d_{+}\right]  z_1^{m_{i}}\cdots \varphi z_1^{m_n}d_{+}
$$
and
$$
D_i=d_{-}z_1^{m_1}\varphi
\cdots z_1^{m_{i-1}}\varphi \left[d_{-}(z_1z_2)^kh_{m_i-k-1}(z_1,z_2)z_1d_{+}\right]\varphi  z_1^{m_{i+1}}\cdots \varphi z_1^{m_n}d_{+}.
$$
By Lemma \ref{lem: right end} we get
$$
Y_{m_1,\ldots,m_n}e_k=(q-1)Y_{m_1,\ldots,m_n,k}+G_{n}-
q^{-1}(1-q)D_{n}.
$$
Note that 
\begin{align*}
G_1&=d_{-}^2T_1^{-1}z_1^{k}d_{+} z_1^{m_{1}}\cdots \varphi z_1^{m_n}d_{+}=
d_{-}^2z_1^{k}d_{+} z_1^{m_{i}}\cdots \varphi z_1^{m_n}d_{+}
=d_{-}z_1^{k}d_{-}d_{+} z_1^{m_{i}}\cdots \varphi z_1^{m_n}d_{+}
\\
&=
d_{-}z_1^{k}d_{+}d_{-} z_1^{m_{i}}\cdots \varphi z_1^{m_n}d_{+}-(q-1)d_{-}z_1^{k}\varphi z_1^{m_{i}}\cdots \varphi z_1^{m_n}d_{+}
\\
&=e_kY_{m_1,\ldots,m_n}-(q-1)Y_{k,m_1,\ldots,m_n}.
\end{align*}
By Corollary \ref{cor: push T left} we get
$$
G_i=G_{i-1}-q^{-1}(1-q)D_{i-1}
$$
and 
$$
G_1-G_n=q^{-1}(1-q)\sum_{i=1}^{n-1}D_i.
$$
Therefore
\begin{align*}
[e_k,Y_{m_1,\ldots,m_n}]&=G_1+(q-1)Y_{k,m_1,\ldots,m_n}-(q-1)Y_{m_1,\ldots,m_n,k}-G_n+q^{-1}(1-q)D_n
\\
&=(q-1)Y_{k,m_1,\ldots,m_n}-(q-1)Y_{m_1,\ldots,m_n,k}+q^{-1}(1-q)\sum_{i=1}^{n}D_i.
\end{align*}
We can expand $D_i$ as a linear combination of $Y$'s by Lemma \ref{lem: Di to Y}, and the result agrees with the right hand side of \eqref{eq: Y relation 2}
up to the sum
$$
(q-1)Y_{k,m_1,\ldots,m_n}-(q-1)Y_{m_1,\ldots,m_n,k}+q^{-1}(1-q)\sum_{i=1}^{n}\left[qY_{m_1,\ldots,k,m_i,\ldots,m_{n}}-qY_{m_1,\ldots,m_i,k,\ldots,m_{n}}\right]=0.
$$
\end{proof}

\bibliographystyle{plain}
\bibliography{bibliography.bib}

\end{document}